\documentclass[11pt]{article}

\usepackage{amsmath,amssymb,amsthm,graphicx}
\usepackage{a4wide}
\usepackage{showlabels}

\newcommand{\Mtilde}{ \widetilde{{\cal M}}  }
\newtheorem{theorem}{Theorem}
\newtheorem{lemma}{Lemma}

\newtheorem{definition}[lemma]{Definition}

\numberwithin{lemma}{section}
\numberwithin{theorem}{section}
\numberwithin{fact}{section}
\numberwithin{equation}{section}

\title{Limit laws for discrete excursions and meanders and linear functional equations with a catalytic variable}

\author{Uwe~Schwerdtfeger\\
\\
RMIT University and The University of Melbourne\\
139 Barry Street, Carlton, VIC 3053, Australia\\
}

\begin{document}
\maketitle

\begin{abstract}
We study limit distributions for random variables defined in terms of coefficients of a power series which is determined by a certain linear functional equation. Our technique combines the method of moments with the kernel method of algebraic combinatorics. As limiting distributions the area distributions of the Brownian excursion and meander occur. As combinatorial applications we compute the area laws for discrete excursions and meanders with an arbitrary finite set of steps and the area distribution of column convex polyominoes. As a by-product of our approach we find the joint distribution of area and final altitude for meanders with an arbitrary step set, and for unconstrained Bernoulli walks (and hence for Brownian Motion) the joint distribution of signed areas and final altitude. We give these distributions in terms of their moments.\\


\end{abstract}

\section{Introduction}
Banderier and Flajolet \cite{BanFla02} give a nice treatment of the analytic combinatorics of directed lattice paths, i.e. walks on the integer line with jumps from a finite set. Their method to obtain the relevant generating functions is known as the \emph{kernel method} \cite{BouPet00,BouJeh06}. They address the asymptotic  enumeration of unrestricted paths, bridges, meanders and excursions and further give limit laws for a couple of counting parameters on lattice paths, such as the final altitude or the number of contacts with the axis. In a later paper \cite{BanGit06} the asymptotic study of the area under a path for excursions and meanders is initiated, by further using these techniques, however with the restriction to step sets containing a single negative step of unit length and the first moment only. This paper continues their work by completing the study of the full area distribution for arbitrary finite step sets. As limit distributions we find the area random variables of Brownian excursion and meander, which for some step sets is actually a simple consequence of weak convergence of the conditioned random walks to their Brownian motion counterparts, see the remarks below Theorem \ref{theo1}. In our combinatorial study we re-prove some of these results and provide the missing cases in the lattice path framework. Furthermore, following a suggestion in \cite{Nguyen03}, we compute the joint distribution of the signed areas and the endpoint of a Brownian motion, which, to the best of our knowledge, are as yet unavailable in the literature. We obtain our results by studying a class of functional equations \eqref{eq:functionalequation} which comprises those equations arising in lattice path and polygon counting \cite{Bousquet96}. This may be of use in different combinatorial contexts.

Tak\'acs \cite{Takacs91,Takacs95} studies Bernoulli excursions and meanders to prove recursion formulas satisfied by the moments of the Brownian excursion area resp. meander area which have proved extremely useful in combinatorial probability. To that end he analyses a functional equation satisfied by the generating function of excursions $E(z,q)$ ($z$ marking length, $q$ area) which reads $E(z,q)=1/(1-zqE(zq,q).$ It allows to derive (univariate) generating functions for the excursion moments, see also \cite{Nguyen03} for a survey. This approach can be generalised to combinatorial classes counted by  a ``size"  and an additional parameter, marked by $z$ and $q,$ whose generating functions satisfy equations of the form $F(z,q)=G(z,q,F(zq,q))$ \cite{Duchon99,Richard09}. This leads eventually to the very same recursions of the limiting moments and finally a Brownian excursion area limit distribution for that parameter. These "$q$-shifts" occur typically for cumulative counting parameters like construction costs of hash tables, internal path length in random trees or area of polynomioes, see e.g. \cite{FlaLou01}.

Such functional equations typically reflect a combinatorial decomposition of the class into smaller objects of the same class.
For lattice paths with a more general set of steps than $\{\pm1\}$ we are lacking a decomposition and functional equation as above. The kernel method has proved to be the weapon of choice for the enumeration by length in \cite{BanFla02}. To that end, an additional ``catalytic"  counting parameter, the final altitude (marked by $u$) is introduced, which allows to turn a simple step-by-step construction of the lattice paths into a linear functional equation for the length-, area- and final altitude generating function $F(z,q,u),$ however, with additional unknown functions in it. The shared feature with the above equation is the $q$-shift, the occurrence of $F(z,q,uq)$ in that equation, see \eqref{eq:functionalequation} below. Very similar equations with $q$-shift occur in different combinatorial contexts, e.g. column convex polyominoes counted by perimeter, area, and height of the rightmost column \cite{Bousquet96}. We prove a limit theorem for a class of functional equations which comprises the above mentioned lattice paths and also the column convex polygons in Section \ref{columnconvex}.

\medskip
\noindent
\textbf{Remark.} Upon the release of a preprint of this work, it has been brought to the author's attention that independently C. Banderier and B. Gittenberger have been continuing their work started in \cite{BanGit06} intermittently over the past years, along similar lines.

\section{Statement of results}\label{sec:meanexcu}

We choose our notation close to the work in \cite{BanFla02}. 
Given a finite subset ${\cal S} \subset\mathbb{Z}$ we consider walks on the integers with $m$ steps from ${\cal S},$ which are simply finite sequences $(w_0,w_1,\ldots,w_m)$ of integers with $w_{i}-w_{i-1}\in {\cal S},$ $i=1,\ldots,m.$ For the rest of the paper we fix $w_0=0.$ The stress of this work is on walks $(w_0,\ldots,w_m)$, $w_i\in\mathbb{Z}_{\geq 0},$ in the non-negative half-line called \emph{meanders}, and on meanders with $w_m=0,$ called \emph{excursions}. In order to make both problems non-void, we assume that there are \emph{positive} integers $c$ and $d,$ such that
\[
-c=\min{\cal S} \textnormal{ and }d=\max{\cal S}.
\]
The generating function of the step set is the Laurent polynomial
\[
S(u)=s_{-c}u^{-c}+\ldots +s_du^d,
\]
where $s_i\ge 0,\; i=-c,\ldots,d$ are non-negative weights (e.g. $s_i\in\{0,1\}$) and  $s_{-c}\neq 0 \neq s_{d}.$ We shall occasionally speak of the mean and variance etc. of the \emph{step set}, by which we mean the corresponding ${\cal S}$-valued random variable with probability generating function $S(u)/S(1).$ 

For a step set ${\cal S}=\{-c=\sigma_1\le \sigma_2\le\ldots \le \sigma_l=d \}$  we define its \emph{period} to be $p:=\gcd(\sigma_2-\sigma_1,\ldots,\sigma_l-\sigma_1).$ The step polynomial can hence be written as $S(u)=u^{-c}H(u^p)$ with a polynomial $H$ of degree $(c+d)/p.$ We call ${\cal S}$ and the resulting walk \emph{aperiodic}, if $p=1.$

The \emph{weight} of a walk $w=(w_0,w_1,\ldots,w_m)$ is the product
\[
\textnormal{wt}(w)=s_{w_1-w_{0}}s_{w_2-w_{1}}\cdots s_{w_{m}-w_{m-1}}
\]
of the weights of the steps of the walk.
The ``graph" $\{(i,w_i)|i=0,\ldots,m\}$ of a walk $(w_0,\ldots,w_m)$ can be viewed as a directed lattice path on $\mathbb{Z}\times\mathbb{Z},$ whose steps are elements of $\{1\}\times{\cal S},$ meanders being directed paths never taking a step below the $x$-axis and excursions additionally ending on that axis, see Figure \ref{fig:meanexcu}. No confusion shall arise by identifying a meander with its graph.
\begin{figure}[htb]
\begin{center}
\includegraphics[height=30mm,width=150mm]{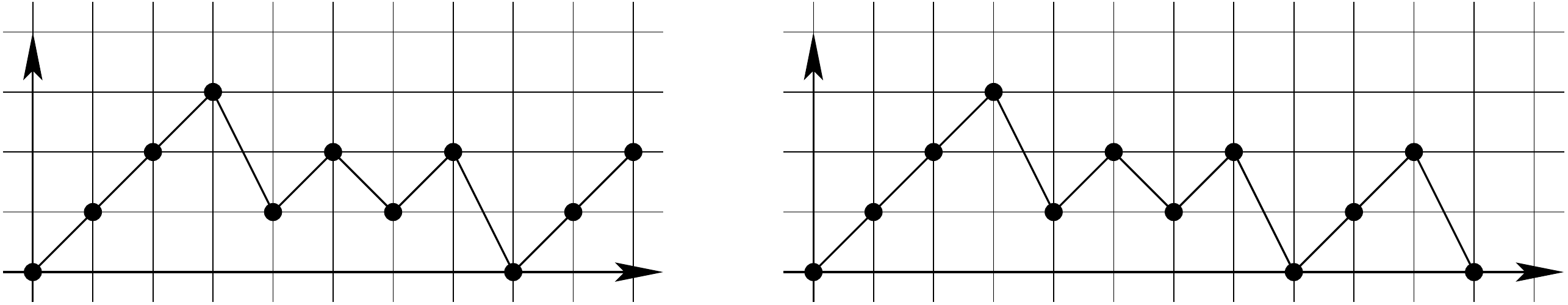}
\caption{A meander and an excursion with steps from ${\cal S}=\{-2,\;-1,\;1 \}.$}\label{fig:meanexcu}
\end{center}
\end{figure}

For a meander $w=(w_0,\ldots,w_m)$ with $m$ steps we consider the functionals 
\[
l(w)=m, \quad h(w)=w_m,\quad a(w)=\sum_{i=0}^mw_i,
\]
which are the length, the terminal altitude and the area between the corresponding directed lattice path and the $x$-axis plus $w_m/2.$ Hence for excursions it is precisely the area. We refer to $a$ in both cases as the \emph{area} below a path. Our main result shows that this abuse of language is justified with view on asymptotics, as the expected area in the considered ensemble turns out to be at least of order $m^{3/2}$ while the final altitude $w_m$ is bounded by $md.$ It proves useful for our approach to define the area random variables in terms of generating functions. Denote by $F(z,q,u)$ be the generating function of meanders enumerated by the number of steps, area and final altitude, marked by $z,$ $q$ and $u,$ respectively, i.e. the power series
\[
F(z,q,u)=\sum_{w}\textnormal{wt}(w)z^{l(w)}q^{a(w)}u^{h(w)},
\]
where $w$ runs through the meanders with steps in ${\cal S}.$ Similarly, let $G_0(z,q)$ ($=F(z,q,0)$) be the generating function of excursions. The area random variable $X_m$ for excursions and $Z_m$ for meanders is defined by
\begin{equation}\label{eq:XmZm}
\mathbb{P}(X_m=l)= \frac{\left[z^mq^l \right]G_0(z,q) }{ \left[z^m \right]G_0(z,1) },\quad   \mathbb{P}(Z_m=l)= \frac{\left[z^mq^l \right]F(z,1,q) }{ \left[z^m \right]F(z,1,1) },
\end{equation}
where the square brackets denote coefficient extraction. Similarly, the random variable of terminal altitude $H_m$ on the set of meanders of length $m$ is given by
\begin{equation}\label{eq:Hm}
 \mathbb{P}(H_m=l)= \frac{\left[z^mu^l \right]F(z,u,1) }{ \left[z^m \right]F(z,1,1) }.
\end{equation}
If, for example, all non-zero weights $s_i$ are equal to $1,$ we consider the area and terminal altitude of a random meander (excursion) drawn uniformly from the set of all meanders (excursions) of length $m.$ 

\subsection*{Brownian motion, excursion and meander}
The main results of this article is expressed in terms of integrals of processes related to Brownian motion. More precisely, let $\left( {\cal B}(t),\;t\in [0,1]\right)$ be a standard Brownian motion of duration 1, ${\cal B}^+(t)=\max \{ {\cal B}(t),0 \}$ and ${\cal B}^-(t)=\min \{ {\cal B}(t),0 \}$ its positive and negative part, respectively. We define the random variables
\[
{\cal A}=\int_0^1\left |{\cal B}(t) \right | {\rm d}t,\quad {\cal A}^-=\int_0^1\left |{\cal B}^-(t) \right | {\rm d}t,\quad {\cal A}^+=\int_0^1{\cal B}^+(t) {\rm d}t
\]
of absolute, negative and positive area, respectively. 
The Brownian meander ${\cal B}^{\textnormal{me}}(t)$ (Brownian excursion ${\cal B}^{\textnormal{ex}}(t)$) can be defined as the process $\left( {\cal B}(t),\;t\in [0,1]\right)$ conditioned on $ {\cal B}(t)\ge 0,$ $t\in [0,1]$ (on $ {\cal B}(t)\ge 0$ and $ {\cal B}(1)=0$), cf. \cite{DurIglMil77}. For a definition in terms of distribution functions we refer to \cite{Takacs91,Takacs95}. However, we only need their area random variables 
\[
{\cal BMA}=\int_0^1{\cal B}^{\textnormal{me}}(t) {\rm d}t,\quad   {\cal BEA}= \int_0^1{\cal B}^{\textnormal{ex}}(t){\rm d}t
\]
the (for our purposes) most convenient definition of which is in terms of their moments.
\begin{definition}[cf. \cite{Takacs91}]\label{def:BE}
The Brownian Excursion Area ${\cal BEA}$ is the random variable defined by the sequence of moments
\begin{equation}
\frac{1}{n!}\mathbb{E}\left( {\cal BEA}^n \right)=\frac{K_n\Gamma(-1/2) 2^{-n/2} }{K_0\Gamma(3n/2-1/2)},
\end{equation}
where the numbers $K_n$ are given by the recursion $K_0=-1/2$ and for $n\geq 1$ by
\begin{equation}\label{eq:Kn}
K_n=\frac{3n-4}{4}K_{n-1}+\sum_{l=1}^{n-1}K_lK_{n-l}.
\end{equation}
\end{definition}
\begin{definition}[cf. \cite{Takacs95}]\label{def:BM}
The Brownian Meander Area ${\cal BMA}$ is the random variable defined by the sequence of moments
\begin{equation}
\frac{1}{n!}\mathbb{E}\left( {\cal BMA}^n \right)=\frac{Q_n\Gamma(1/2) 2^{-n/2} }{Q_0\Gamma(3n/2+1/2)},
\end{equation}
where the numbers $Q_n$ are given recursively by $Q_0=1$ and for $n\geq 1$ by
\begin{equation}\label{eq:Qn}
Q_n=\frac{3n-2}{2}Q_{n-1}+2\sum_{l=1}^{n}K_lQ_{n-l}.
\end{equation}
The numbers $K_n$ are given in \eqref{eq:Kn}.
\end{definition}
\noindent
\textbf{Remark.} ${\cal BEA}$ and ${\cal BMA}$ are well defined by their moments, cf. \cite{Simon98}. 

\medskip
\noindent
With all the vocabulary at hand we state our results.
\begin{theorem}\label{theo1}
Let ${\cal S}$ be the step set of an aperiodic (see definition below) walk with step polynomial $S(u)$ and drift $\gamma=S^\prime(1)/S(1).$ Let furthermore $\tau$ be the unique positive zero of $S^\prime(u)$ and $\beta=\sqrt{2S(\tau)/S^{\prime\prime}(\tau)}.$ The sequence of rescaled discrete meander area random variables of meanders with steps in ${\cal S}$ admits a limit law depending on the sign of the drift. More precisely we have for negative drift $\gamma<0$ the weak limit
\begin{equation}\label{eq:negdrift}
\frac{\beta Z_m}{\sqrt{2}m^{3/2}}\stackrel{d}{\longrightarrow}{\cal BEA}.
\end{equation}  
For zero drift $\gamma=0$ the joint distribution of the rescaled area and terminal altitude random variable converges weakly to the joint distribution of  $( {\cal BMA},{\cal B}^{\textnormal{me}}(1)),$ more precisely
%
%
%
\begin{equation}\label{eq:zerodrift}
 \left(\frac{\beta Z_m}{\sqrt{2}m^{3/2}} , \frac{\beta H_m}{\sqrt{2}m^{1/2}}\right) \stackrel{d}{\longrightarrow} ( {\cal BMA},{\cal B}^{\textnormal{me}}(1)).
\end{equation}
In both cases we have moment convergence. In the case of positive drift $\gamma >0$ we find the expected area to be asymptotically equal to $\gamma m^2/2$ while the standard deviation is of order $o\left(m^{2}\right).$ Hence the sequence is concentrated to the mean,
\begin{equation}\label{eq:posdrift}
\frac{ Z_m}{\gamma m^{2}/2}\stackrel{\text{prob.}}{\longrightarrow} 1.
\end{equation}  
The sequence of rescaled discrete excursion area random variables admits a weak limit law independent of $\gamma,$ namely 
\begin{equation}\label{eq:excursionlaw}
\frac{\beta X_m}{\sqrt{2}m^{3/2}}\stackrel{d}{\longrightarrow}{\cal BEA}.
\end{equation}  
We also have moment convergence.
\end{theorem}

\noindent
\textbf{Remark.} \textit{i)} Notice that in the case of zero drift $\tau=1$ and $\beta/\sqrt{2}$ is the reciprocal of the standard deviation of the step set, in accordance with the rescaling of the random walks in the functional central limit theorem in \cite{Iglehart74}. Hence \eqref{eq:zerodrift} follows from \cite{Iglehart74} by the continuity of the integral and evaluation operator $f \mapsto f(1).$ 

\noindent
\textit{ii)} For $\gamma>0$ the rescaled and ``centered" discrete meander converges weakly to ${\cal B}$ \cite{Iglehart74} and furthermore $\int_0^1{\cal B}(t)dt $ is ${\cal N}(0,1/3)$-distributed. For the area we hence have
\[
\frac{Z_m-\gamma m^2/2}{\sigma m^{3/2}} \stackrel{\text{d}}{\longrightarrow} {\cal N}(0,1/3),
\]
where $\sigma^2=\left[S^{\prime\prime}(1)+S^{\prime}(1)\right]/S(1)-\gamma^2$ is the variance of the step set and ${\cal N}(\mu,\theta^2)$ is the normal distribution with mean $\mu$ and variance $\theta^2.$

\noindent
\textit{iii)} Relations \eqref{eq:negdrift} and \eqref{eq:excursionlaw} do not seem to follow in this way: e.g. in \cite{Kaigh76} it is shown that a \emph{zero drift} random walk of length $n$ conditioned on having its \emph{first return to zero} at time $n$ converges weakly to the Brownian excursion. However, this conditioning does not imply that the walk is all positive or all negative prior to time $n,$ since it may jump across $0$ without actually hitting $0$ (unless the step set is $\subseteq \{-1,0,1\}$). Tying down a random walk as in \cite{DurIglMil77} is not a continuous operation. Regarding negative drift meanders, we have only found weak convergence results for \emph{non-lattice} step sets in the literature \cite{Durrett80, Kao78}. \qed 

\medskip
\noindent
We formulate a more general result in Theorem \ref{generaltheo}, which Theorem \ref{theo1} is a special case of.

%
\begin{theorem}\label{theo:BrownianMeander}
\textit{i)} The moments of ${\cal BEA}$ allow the alternative description
\begin{equation}
\frac{1}{n!}\mathbb{E}\left( {\cal BEA}^n \right)=\frac{-\Gamma(-1/2)  }{\Gamma(3n/2-1/2)} 2^{ -\frac{7n-2}{2} } C_{n-1,1},
\end{equation}
where the numbers $C_{n,t}$ are defined by $C_{n,t}=0,$ if $t<0$ or $n<0,$ $C_{0,0}=1$ and else by
\begin{equation}
C_{n,t}=C_{n,t-1}+(t+2)C_{n-1,t+2}.
\end{equation}
Equivalently, $C_{n-1,1}=8^nK_n,$ with $K_n$ as in Definition \ref{def:BE} of ${\cal BEA}.$ 

\noindent
\textit{ii)} The joint distribution of the area ${\cal BMA}$ and the final altitude ${\cal B}^{\textnormal{me}}(1)$ of the Brownian Meander is uniquely determined by the sequence of moments
\begin{equation}
M_{n,t}=\mathbb{E} \left[  {\cal BMA}^n \;  {\cal B}^{\textnormal{me}}(1)^t \right] =\frac{   n!t! \Gamma(1/2) 2^{-(n+t)/2} }{  \Gamma(3n/2+t/2+1/2)  }Q_{n,t},
\end{equation}
where the numbers $Q_{n,t}$ are given by 
\begin{equation}\label{eq:Qnt}
Q_{n,t}=
\begin{cases}
Q_{n,t-2}+(t+1)Q_{n-1,t+1},\; \textnormal{if }t\ge 1,\\
Q_{n-1,1}+ 2\cdot 8^{-n}C_{n-1,1},\;\textnormal{if }t=0,
\end{cases}
\end{equation}
with the initial values $Q_{0,0}=Q_{0,1}=1$ and $Q_{n,t}=0$ if $t<0$ or $n<0.$ In particular, we have $Q_{n,0}=Q_n,$ cf. Definition \ref{def:BM}, and $M_{0,t}$ is the $t$th moment of the Rayleigh  distribution on $[0,\infty)$ given by the distribution function $1-\exp(-x^2/2).$  
\end{theorem}
\noindent
%
\textbf{Remark.} The structure of the proof of the theorems \ref{theo1} and \ref{theo:BrownianMeander} is as follows. From the functional equation \eqref{eq:functionalequation} we derive moment convergence to limiting moments which satisfy  recursion relations equivalent to those in Theorem \ref{theo:BrownianMeander}. It is only in Section \ref{sec:ProofLemmas} where we make use of the fact that the \emph{Bernoulli} excursions and meanders converge weakly to their Brownian siblings, in order to prove Theorem \ref{theo:BrownianMeander}. Thus, assuming Theorem \ref{theo:BrownianMeander} is given, the limit laws in Theorem \ref{theo1} can be viewed as a consequence of the functional equation only.
%
%
%


\medskip
\noindent
The next theorem is about the joint distribution of signed areas and terminal altitude of a Brownian Motion. They follow from the convergence of the respective Bernoulli walks to their Brownian motion counterparts, cf. \cite{Drmota04,Nguyen03,Takacs91,Takacs95}.  

\begin{theorem}\label{theo:signedAreas}
The joint distribution of the signed areas ${\cal A}^-,$ ${\cal A}^+$ and the terminal altitude ${\cal B}(1)$ is uniquely determined by its joint moments given by
\begin{equation}
\mathbb{E}\left[ ({\cal A}^+)^k({\cal A}^-)^l  {\cal B}(1)^t\right] =\frac{k!l!t!2^{-(k+l+t)/2}}{\Gamma(3k/2+3l/2+t/2+1) } L^{\pm}_{k,l,t}.
\end{equation}
The numbers $L^{\pm}_{k,l,t}$ are given recursively by $L_{0,0,2t} =1$ and $L_{0,0,2t+1}=0,$    $t \ge 0,$ and for $(k,l)\neq (0,0)$ by
\begin{equation}
L^{\pm}_{k,l,t}=\sum_{j=1}^kK_jL^{\pm}_{k-j,l,t} +\sum_{j=1}^lK_jL^{\pm}_{k,l-j,t}+\frac{1}{2}\delta_{l,0}Q_{k,t}+\frac{1}{2}\delta_{k,0}Q_{l,t}.
\end{equation}

Similarly, the joint distribution of the absolute area ${\cal A}$ and ${\cal B}(1)$ is uniquely defined by the moment sequence
\begin{equation}
\mathbb{E}\left( {\cal A}^n  {\cal B}(1)^t\right) =\frac{n!t!2^{-(n+t)/2}}{\Gamma(3n/2+t/2+1) } L_{n,t},
\end{equation}
where $L_{n,2t+1}=0$ for $n,t\ge 0,$ $L_{0,2t}=1$ for $t\ge0$ and for $n\ge 1$ $L_{n,2t}$ is given by the recursion
\begin{equation}
L_{n,2t}=2\sum{K_jL_{n-j,2t}}+Q_{n,2t}.
\end{equation}
The numbers $K_n$ are given in Definition \ref{def:BE} and the numbers $Q_{n,t}$ in Theorem \ref{theo:BrownianMeander}.
\end{theorem}

\section{The fundamental functional equation}

Our proofs are by the method of moments based on generating functions. Recall that the joint factorial moments of $Z_m$ and $H_m$ are given by 
\[
\mathbb{E}((Z_m)_n(H_m)_t)=\frac{ \left[z^m \right]\left . \left(\frac{\partial}{\partial q}\right)^n\left(\frac{\partial}{\partial u}\right)^tF(z,q,u) \right |_{q=u=1}  }{\left[z^m \right]F(z,1,1)},
\]
where $(a)_n=a(a-1)\cdots(a-n+1)$ denotes the falling factorial, similarly for $X_m.$ So, moments are expressed in terms of coefficients of univariate power series, which are in turn amenable to the process of singularity analysis. The required moment generating functions are ``pumped" from a modified version of the fundamental functional equation of lattice path enumeration  \cite{BanFla02}, 
\begin{equation}
F(z,q,u)=1+zS(uq)F(z,q,uq)-z\sum_{i=0}^{c-1}r_i(uq)G_i(z,q).
\end{equation}
where $G_i(z,q)$ is the length and area generating function of meanders with terminal altitude $i$ and $r_i(u)$ is the Laurent polynomial given by
\begin{equation}\label{eq:ri}
r_i(u)=u^i\left(s_{-c}u^{-c}+\ldots+s_{-(i+1)}u^{-(i+1)}\right).
\end{equation}
This equation reflects the combinatorial decomposition ``a meander is either \emph{i)} empty or it is \emph{ii)} obtained by adding a step to a meander \emph{iii)} without going below the $x$-axis".
The given functional equation is simply the $q$-shift of the fundamental functional equation for meanders in \cite{BanFla02}, i.e. $u$ is simply substituted by $uq.$ The proof is similar to that in \cite{BanGit06} where the authors study the actual area (without our notational abuse). 
\medskip

\noindent
\textbf{Notation.} For a commutative ring $R$ and formal (commuting) variables $a_1,a_2,\ldots,a_n$ we denote by $R[a_1,a_2,\ldots,a_n]$ the ring of polynomials, by $R(a_1,a_2,\ldots,a_n)$ the field of fractions thereof, i.e. rational functions, by $R[[a_1,\ldots,a_n]]$ the ring of formal power series.
\medskip

\noindent
In what follows we study a slightly more general functional equation
\begin{equation}\label{eq:functionalequation}
F(z,q,u)=W(z,uq)+zS(z,uq)F(z,q,uq)-\sum_{i=0}^{c-1}r_i(z,uq)G_i(z,q),
\end{equation}
where $W(z,u),$ $S(z,u)$ and the $r_i(z,u)$ are rational functions $\in \mathbb{R}(z,u).$ Denote by $Q(z,u)$ a fixed least common multiple of their denominators.
We will have to take derivatives w.r.t. $q$ and $u$ of this equation, those w.r.t. $q$ causing more trouble, since the multivariate chain rule \cite{ConSav96} is involved. More precisely, for a function $F(x_1,x_2)$ in two variables we need the formula
\[
\left(\frac{\partial}{\partial q}\right)^nF(q,uq)=\sum_{t=0}^n {n\choose t} u^{n-t} \left(\frac{\partial}{\partial x_2}\right)^{n-t}\left(\frac{\partial}{\partial x_1}\right)^t(F)(q,uq),
\]
which is easily shown by induction.
To make dealing with the derivatives of the functional equation a little less messy, we write for the derivatives of $F(u,q)$
\begin{equation}
F^{n,t}(u):=F^{n,t}(z,u):=\left.\frac{\partial^{n+t}}{\partial q^n\partial u^t}F(z,q,u)\right |_{q=1}
\end{equation}
and
\begin{equation}\label{eq:G_i}
G^{(n)}_i:=G^{(n)}_i(z):=\left.\left(\frac{\partial}{\partial q}\right)^nG_i(z,q)\right |_{q=1}.
\end{equation}
\medskip

\noindent
\textbf{Notation:} 
For bivariate functions $T(z,y),$ we denote the partial derivative w.r.t. the second variable $y$ by $T^{\prime}(z,y),$ $T^{\prime\prime}(z,y),\ldots,T^{(n)}(z,y).$ 
\medskip

\noindent
An application of the above special case of the chain rule, setting $q=1$ and regrouping terms eventually yields the following version of the $n$th $q$-derivative of the functional equation evaluated at $q=1:$
\begin{equation}\label{eq:derfcteq}
\begin{split}
(1-&zS(z,u))F^{n,0}(u)+\sum_{i=0}^{c-1}r_i(z,u)G_i^{(n)}=zS(z,u)nF^{n-1,1}(u)\\
+&zS(z,u)\sum_{t=2}^n {n\choose t} F^{n-t,t}(u)\\
+&z\sum_{l=1}^n  \sum_{t=0}^{n-l}{n\choose l} {n-l\choose t}u^{l+t}S^{(l)}(z,u) F^{n-l-t,t}(u)\\
-&z\sum_{i=0}^{c-1}  \sum_{l=1}^{n}{n\choose l}u^lr^{(l)}(z,u)G_i^{(n-l)}+u^nW^{(n)}(z,u)=:RHS_n(z,u).
\end{split}
\end{equation}
%
%
If we multiply this equation by $Q(z,u)$ the coefficients of $F^{n,0}(u)$ and the $G_i^{(n)}$ on the left hand side are polynomials. 
It will turn out that for the asymptotic considerations only the terms in the top line of the previous equation are of interest. 
Further derivatives of equation \eqref{eq:derfcteq} w.r.t. $u$ yield an equation for $F^{n,t}(u).$ 

The following sections are dedicated to the singular behaviour of the functions $G^{(n)}_k(z),$ $F^{n,t}(z,1)$ and $F^{n,t}(z,u_1(z)),$ where $u_1(z)$ is a branch of a certain algebraic equation. 
\section{On the solution to the functional equation}

For a self-contained exposition, we sketch how equation \eqref{eq:derfcteq} can be solved, assuming that the rhs is known. We then have $c+1$ unknowns, and hence a seemingly underdetermined system. The \emph{kernel method} \cite{BanFla02,BouPet00,BouJeh06} allows to solve this equation. 

We shall frequently deal with a $c+1$-tuple $(T(z,u),x_0(z),\ldots,x_{c-1}(z))$ of formal power series $T(z,u)\in\mathbb{C}[u][[z]]$ and $x_0(z),\ldots,x_{c-1}(z)\in \mathbb{C}[[z]]$ defined by a linear functional equation
\begin{equation}\label{eq:generalfeq}
\Phi(z,u)T(z,u)+\sum_{i=0}^{c-1}t_i(z,u)x_i(z)=y(z,u),
\end{equation}
where $\Phi(z,u)$ and $t_i(z,u)$ are the polynomials
\[
\Phi(z,u)=Q(z,u)(1-zS(z,u)),\quad t_i(z,u)=Q(z,u)r_i(z,u).
\]

In our framework $y(z,u)\in \mathbb{C}[u][[z]]$ is a power series in $z$ with its coefficients polynomials in $u.$ The $t_i(z,u)$ are assumed to be linearly independent over $\mathbb{C}(z),$ the field of rational functions in $z.$ The degree of $\Phi(z,u)$ in $u$ is assumed to be $c+d$ with $c,d>0.$ We further assume that $\Phi(0,u)$ has degree $c$ and hence there are $c$ fractional power series $u_1(z),\ldots,u_c(z)$ and $d$ fractional Laurent series $v_1(z),\ldots,v_d(z)$ with non-trivial principal part (each counted with multiplicities), such that $\Phi(z,u_i(z))=0=\Phi(z,v_j(z)),$ see \cite[Theorem 2]{BouJeh06} or \cite{Hille62}.  We refer to the $u_i(z)$ and $v_j(z)$ as the small and large branches, respectively, and assume $u_1(z),\ldots,u_c(z)$ to be distinct.


%
\begin{definition}
Let $x_0,\ldots, x_n$ be formal variables, and $f(x)$ a formal power series.
\begin{enumerate}
\item The divided differences are defined recursively as
\[
f[x_i]=f(x_i),\quad f\left[x_{n_1},\ldots, x_{n_k}  \right]=\frac{f\left[x_{n_1},\ldots, x_{n_{k-1}}\right] -f\left[x_{n_2},\ldots, x_{n_k}\right] }{x_{n_1}-x_{n_k}}.
\]

\item For the Vandermonde determinant in the variables $x_1,\ldots, x_n$ we write
\[
\Delta_{n}(x_1,\ldots, x_n)=\prod_{1\leq i < j \leq n}(x_i- x_j).
\]
\end{enumerate}
\end{definition}

If $f(x)$ is a formal power series in $x,$ then $f[x_0,\ldots,x_n]$ is a formal power series in $x_0,\ldots,x_n,$ and furthermore, if $x_0=\ldots=x_n$ then $f[x_0,\ldots,x_n]=\frac{1}{n!}f^{(n)}(x_0).$ Furthermore, if $f$ is a row vector of length $n+1,$ whose entries are formal power series, then the determinant of the $n+1\times n+1$ matrix with rows $f(x_i)$ factorises as 
\[
\det \left( \begin{array}{c}
f(x_0)\\
f(x_1)\\
\vdots \\
f(x_n)
\end{array}\right)= \det \left( \begin{array}{c}
\vdots\\
f[x_{i_0}]\\
f[x_{i_0},x_{i_1}]\\
\vdots \\
f[x_{i_0},\ldots,x_{i_k}]\\
\vdots\\
\end{array}\right) \Delta_{k+1}(x_{i_0},\ldots,x_{i_k}),
\]
where on the rhs the rows not involving $x_{i_0},\ldots,x_{i_k}$ are unchanged. Notice that the determinant on the rhs is a \emph{symmetric function} in $x_{i_0},\ldots,x_{i_k}$ \cite{Bressoud99}.

\medskip
\noindent
Define the matrices
\[
{\cal M}=\left[
  \begin{array}{cccc}
 t_0(z,u_1)¥ & \dots¥ & t_{c-1}(z,u_1)¥ \\ 
    \vdots  & ¥ & ¥\vdots \\ 
 t_0(z,u_c)¥ & \dots¥ & t_{c-1}(z,u_c)¥ \\ 
  \end{array}
\right],\;{\cal N}(z,u;y)= \left[
  \begin{array}{cccc}
    ¥y(z,u) & t_0(z,u)¥ & \dots¥ & t_{c-1}(z,u)¥ \\ 
      ¥y(z,u_1) & t_0(z,u_1)¥ & \dots¥ & t_{c-1}(z,u_1)¥ \\ 
    \vdots &  \vdots¥ & ¥ & ¥\vdots \\ 
      ¥y(z,u_c) & t_0(z,u_c)¥ & \dots¥ & t_{c-1}(z,u_c)¥ \\ 
  \end{array}
\right],
\]
We assume ${\cal M}$ to be invertible. By $\Mtilde$ we denote the matrix
\[
\widetilde{{\cal M}}=\left[\,t_j(z,\cdot )[u_1,\ldots,u_i]\,  \right ]_{\begin{smallmatrix}
i=1,\ldots,c\\
j=0,\ldots,c-1 
\end{smallmatrix}}
\]
and by 
\[
{\cal G}_k(y)=\left[
  \begin{array}{ccccccc}
 t_0(z,u_1)¥ & \dots¥ &t_{k-1}(z,u_1)¥ & y(z,u_1) &t_{k+1}(z,u_1) &\dots¥  &t_{c-1}(z,u_1)¥ \\ 
    \vdots  & ¥& \vdots &  \vdots& \vdots ¥& &\vdots \\ 
 t_0(z,u_c)¥ & \dots¥ &t_{k-1}(z,u_c)¥ & y(z,u_c) &t_{k+1}(z,u_c) &\dots¥  &t_{c-1}(z,u_c)¥
  \end{array}\right]
\]

and finally by ${\cal M}(u;i)$ (resp. $\Mtilde(u;i)$) the matrix obtained from ${\cal M}$ (resp. $\Mtilde$) by substituting $u_i$ by $u.$ Notice that $\det \Mtilde$ and $\det \Mtilde(u;i)$ are the \emph{symmetric polynomials} obtained by division of $\det {\cal M}$ (resp. ${\cal M}(u;i)$) by the respective Vandermonde determinants.

\begin{lemma}
The functional equation \eqref{eq:generalfeq} has the unique power series solution with $T(z,u)$ given by
\begin{equation}\label{eq:generalsol}
\begin{split}
T(z,u)&=\frac{1}{\Phi(z,u)}\left(y(z,u)-\sum_{l=1}^c y(z,u_l(z))\frac{\det{\cal M}(u;l) }{\det {\cal M}}\right)\\
&=\frac{1}{\Phi(z,u)}\left(y(z,u)-\sum_{l=1}^c y(z,u_l(z))\frac{\det{\Mtilde}(u;l) }{\det {\Mtilde}}   \prod_ {\begin{smallmatrix}  j=1\\ j\neq l \end{smallmatrix}    } ^c  \frac{u-u_j}{u_l-u_j}\right)\\
&=\frac{ \det {\cal N}(z,u;y)  }{  \Phi(z,u) \det{\cal M} }
\end{split}
\end{equation}
and the power series $x_k(z)$ given
\begin{equation}\label{eq:Gk}
\begin{split}
x_k&=\frac{\det {\cal G}_k(y)}{\det {\cal M}}\\
&=\frac{(-1)^{k+c-1} }{ s_{-c}z }\sum_{l=1}^cy(z,u_l)\frac{\phi_k(u_1,\ldots,u_{l-1},u_{l+1},\ldots,u_c)}{\det \Mtilde}
\prod_ {\begin{smallmatrix}  j=1\\ j\neq l \end{smallmatrix}    } ^c  \frac{1}{u_l-u_j}, 
\end{split}
\end{equation}
where $\phi_k(z_1,\ldots,z_{c-1})$ is a symmetric polynomial in $c-1$ variables. 

\end{lemma}
\begin{proof}
We drop the variable $z$ and write $y(u)$ instead of $y(z,u)$ etc.. 
Substituting $u$ by $u_1,\ldots,u_c$ eliminates $T(z,u)$ from \eqref{eq:generalfeq} and yields $c$ equations for the $x_i,$ namely the system
\[
{\cal M} (x_0,\ldots,x_{c-1})^T=(y(u_1),\ldots,y(u_c))^T,
\]
By Cramer's rule and a Laplace expansion along the $k$th column ($k=0,\ldots ,c-1$) we have
\begin{equation}\label{eq:xk}
x_k=\frac{\det {\cal G}_k(y(u))}{\det {\cal M}}=\sum_{l=1}^c y(u_l)(-1)^{k+l-1}\frac{\det {\cal M}_{k,l}}{\det {\cal M}},
\end{equation}
where ${\cal M}_{k,l}$ is obtained from ${\cal M}$ by deleting the $l$th row and the $k$th column. Since $\det {\cal M}_{k,l}$ is an \emph{alternating} polynomial in $u_1,\ldots,u_{l-1},$ $u_{l+1},\ldots, u_c,$ it can be factorised as
\[
\det {\cal M}_{k,l}=\Delta_{c-1}(u_1,\ldots,u_{l-1},u_{l+1},\ldots,u_c)\,\phi_k(u_1,\ldots,u_{l-1},u_{l+1},\ldots,u_c)
\]
with a symmetric polynomial $\phi_k$ \cite{Bressoud99}, which implies \eqref{eq:Gk}.

As for the representation of $T(z,u)$ we can write
\begin{equation}\label{alternative}
\begin{split}
\sum_{k=0}^{c-1}t_k(u)x_k =\sum_{k=0}^{c-1}t_k(u)\frac{1}{\det {\cal M}}\sum_{l=1}^c y(u_l)(-1)^{k+l-1}\det {\cal M}_{k,l}\\
=\sum_{l=1}^c y(u_l)\frac{1}{\det {\cal M}}\sum_{k=0}^{c-1}t_k(u)(-1)^{k+l-1}\det {\cal M}_{k,l}=\sum_{l=1}^c y(u_l) \frac{\det{\cal M}(u;l)}{\det {\cal M}}.
\end{split}
\end{equation}
This settles the first line of equation \eqref{eq:generalsol}. The third line follows from a Laplace expansion of $\det {\cal N}(z,u;y)$ along the first column, and the second line from the
\end{proof}

\noindent
\textbf{Remark.} The given representations of the $x_k$ later allow to conveniently estimate the singular behaviours of the $G_k^{(n)}.$\\
\medskip

\noindent
In Sections \ref{sec:excarea} and \ref{sec:negdrift} we will need two further identities, whose proofs necessitate the following lemma.
\begin{lemma}[Grassmann-Pl{\"u}cker relations {\cite[p. 34]{Murota00}}]
Let $A\in \mathbb{C}^{n\times m}$ an $n\times m$ matrix with $m<n.$ Let $J,J^{\prime}\subset\{1,\ldots,n \}$ be index sets with $|J|=|J^{\prime}|=m$ and $k\in  J \setminus J^{\prime}.$ Denote by $A_J$ the submatrix consisting of the rows with indices in $J$ and for $t\in J$ and an index $l,$ $A_{J-t+l}$ denotes the matrix obtained by replacing the row with index $t$ by the one with index $l$ in $A_J.$ Then
\[
\det A_{J}\det A_{J^{\prime}} =\sum_{l\in J^{\prime}\setminus J } \det A_{J-t+l}\det A_{J^{\prime}-l+t}.
\] \qed
\end{lemma}
\begin{lemma}
If we consider $u,u_1,\ldots,u_c$ as formal variables, then we have the following identities
\begin{equation}\label{eq:Plucker1}
\det {\cal M}_{k,1}\left.\frac{\partial}{\partial u} \det  {\cal N}(z,u;y(z,u))\right |_{u=u_1} =
\det {\cal M} \frac{\partial}{\partial u_1}\det {\cal G}_k -\det {\cal G}_k\frac{\partial}{\partial u_1}\det {\cal M},
\end{equation}
\begin{equation}\label{eq:Plucker2}
\begin{split}
&\det {\cal M}(u;l) \left.\frac{\partial}{\partial u} \det  {\cal N}(z,u;y(z,u))\right |_{u=u_1} \\
=& \det  {\cal N}(z,u;y(z,u))\frac{\partial}{\partial u_1}\det {\cal M}-\det {\cal M}\frac{\partial}{\partial u_1}\det  {\cal N}(z,u;y(z,u)). \end{split}
\end{equation}
\end{lemma}

\begin{proof}
On the lhs of \eqref{eq:Plucker1} we apply two Laplace expansions to the $(c+1) \times (c+1)$ determinants. First expand along the leftmost column (containing $y^\prime(z,u_1),$ $y(z,u_i)$) and then expand the resulting $c \times c$ determinants along the column containing the $t_k(u_i).$ On the rhs expand each determinant along the row containing the $t_k(z,u_i)$ and $y(z,u_i),$ respectively. The resulting representations are sums of terms of the form $y(z,u_l)t_k(z,u_j)\det A_K \det A_L$ where $A_K$ and $A_L$ are suitable submatrices of the $(c+1)\times (c-1)$ matrix
\[
A:=\left[
  \begin{array}{cccccc}
t_0^{\prime}(z,u_1)¥ & \dots¥ &t_{k-1}^{\prime}(z,u_1)¥  &t_{k+1}^{\prime}(z,u_1)¥   &\dots¥ & t_{c-1}^{\prime}(z,u_1)¥ \\ 
 t_0(z,u_1)¥ & \dots¥ & t_{k-1}(z,u_1)¥ & t_{k+1}(z,u_1)¥ &\dots & t_{c-1}(z,u_1)¥ \\ 
  \vdots¥ &       ¥ & ¥\vdots ¥ & ¥\vdots & ¥  & ¥\vdots \\ 
 t_0(z,u_c)¥ & \dots¥ & t_{k-1}(z,u_c)¥ & t_{k+1}(z,u_c)¥ &\dots & t_{c-1}(z,u_c)¥ \\   \end{array}
\right].
\]
Let the rows be indexed with $I=\{-1,1,\ldots,c\}$ from top to bottom, and for $K\subset I$ denote by $A_K$ the submatrix consisting of the rows indexed by $K.$ We prove the identity by comparing the coefficients of $y(z,u_l)t_k(z,u_j),$  $y^\prime(z,u_1)t_k(z,u_l),$ $y(z,u_l)t_k^\prime(z,u_1)$ and $y^\prime(z,u_1)t_k^\prime(z,u_1)$ $j,l=1\ldots,c$ in both expansions. The coefficient of $y(z,u_l)t_k(z,u_j),$ $l,j\in I$  on the lhs is
\[
(-1)^{l+ \epsilon(j)} \det A_{I\setminus \{l,j \}}\det A_{I\setminus \{-1,1 \}}
\]
and on the rhs
\[
(-1)^{l+ \epsilon(j)} \left(\det A_{I\setminus \{-1,l \}}\det A_{I\setminus \{1,j \}}  -\det A_{I\setminus \{l,1 \}}\det A_{I\setminus \{-1,j \}} \right),
\]
where $\epsilon(j)\in\{0,1\}$ and we interpret the terms for $j=-1$ or $l=-1$ accordingly. Let $J=I\setminus \{-1,1\},$ $J^\prime=I\setminus \{j,l\}.$ If $j,l\ge 2,$ then $J\setminus J^\prime=\{j,l\},$ $J^\prime\setminus J=\{-1,1\}.$ If we apply the Grassmann-Pl\"ucker relations with $t=j\in J\setminus J^\prime$ the equality follows, observing that, for example, 
\[
\det A_{I\setminus \{-1,l \}}\det A_{I\setminus \{1,j \}}=\det A_{J^\prime-(-1)+j}\det A_{J-j+(-1)},
\]
which is seen by elementary row operations. If $j\in \{-1,1\}$ or $l\in \{-1,1\}$ one argues similarly (one of the summands on the rhs becomes zero).

To prove the identity \eqref{eq:Plucker2}, we expand the $(c+1)\times (c+1)$ determinants along the first column and compare the respective coefficients of $y(z,u_i)$ and $y^{\prime}(z,u_1)$ on both sides. For $Y(z,u_1)$ and $Y^{\prime}(z,u_1)$ the equality is apparent, for $i\ge 2$ we apply the Grassmann-Pl\"ucker relations to the $(c+2)\times c$ matrix
\[
\left[
  \begin{array}{ccc}
t_0^{\prime}(z,u_1)¥ & \dots¥ & t_{c-1}^{\prime}(z,u_1)¥ \\ 
 t_0(z,u)¥ & \dots¥ & t_{c-1}(z,u)¥ \\ 
 t_0(z,u_1)¥ & \dots¥ & t_{c-1}(z,u_1)¥ \\ 
  \vdots¥ & ¥ & ¥\vdots \\ 
 t_0(z,u_c)¥ & \dots¥ & t_{c-1}(z,u_c)¥ \\ 
  \end{array}
\right].
\]
Here we we consider the rows indexed with the set  $I=\{-1,0,1,\ldots,c\}$ (from top to bottom), $J=I\setminus \{0,i\},$ $J^{\prime}=I\setminus \{-1,1\}$ and $t=-1.$

\end{proof}

\section{Evaluation at the small branches}
In this section we start the study of the singular behaviour of the functions $F^{n,t},$ which can be computed recursively with the method of the previous section from the functional equation \eqref{eq:functionalequation}
\begin{equation*}
F(z,q,u)=W(z,uq)+zS(z,uq)F(z,q,uq)-\sum_{i=0}^cr_i(z,uq)G_i(z,q).
\end{equation*}
The following assumptions imply that the singular behaviour of the $F^{n,t}$ and $G^{(n)}$ originates in that of a single small branch of the kernel equation. 
\subsection{Analytic assumptions}\label{sec:analyticassumptions}

Recall that $\Phi(z,u)=\text{lc}(z)(u-u_1(z))\cdots (u-u_c(z))(u-v_1(z))\cdots (u-v_c(z))$ with $\text{lc}(z)$ the coefficient of $u^{c+d}.$ We have to make some assumptions on the analytic behaviour of the branches $u_i,$ $v_j,$ or, equivalently, the singular points of the kernel equation $\Phi(z,u)=0.$  

\begin{enumerate}
\item \emph{Small branches are distinct:} $\Phi(0,u)$ has degree $c$ and the multiplicities of the small branches $u_1,\ldots,u_c$ are one.
\item \emph{Structural radius $\rho$ and square root behaviour:} There exists a point $(\rho,\tau),$ $0<\rho,$ $0<\tau,$ with  $1-\rho S(\rho,\tau)=S^{\prime}(\rho,\tau)=0$  ( $\Rightarrow \Phi(\rho,\tau)=\Phi^{\prime}(\rho,\tau)=0$). Additionally we assume $S^{\prime\prime}(\rho,\tau)> 0$  (i.e. $\Phi^{\prime\prime}(\rho,\tau)\neq 0$). 
The equation $z=1/S(z,u(z))$ has hence two (real) solutions $u^\pm(z)$ close to $(z,u)=(\rho,\tau)$ with Puiseux expansions 
\begin{equation}\label{eq:genu1}
u^{\pm}(z)=\tau\pm\beta \sqrt{1-z/\rho}+O(1-z/\rho)\quad(z\longrightarrow \rho), \quad \beta=\sqrt{2\frac{S(\rho,\tau)}{S^{\prime\prime}(\rho,\tau)}}.
\end{equation}
We assume that $u^\pm$ can be continued analytically along $(0,\rho)$ and that $u^-$ has a finite limit at $z=0,$ while $u^+$ becomes infinite. Hence $u^-$ is a small and $u^+$ a large branch, denoted by $u_1(z)$ and $v_1(z),$ respectively.
\item \emph{Separation of small and large branches and uniqueness of $\rho$:} We assume that $z=\rho$ is the only point  in $\{z:|z|\leq\rho\}$ for which $u_1(z)=v_1(z).$ Furthermore 
\[
\forall\, z\in \{w:|w|\leq\rho\}\; \forall \, i\in\{2,\ldots, c\} \; \forall\, j\in\{2,\ldots, d\}:\quad u_i(z)\neq v_j(z).
\]
\item We assume that $\text{lc}(z)\neq 0$ and $Q(z,u_i(z))\neq 0$ in $\{ z:|z| \leq \rho \}.$
\item \emph{Zeroes of the determinant:} Let $g(z)$ be the greatest common divisor of the coefficients of $t_{0}(z,u),\ldots,t_{c-1}(z,u)$ viewed as polynomials in $u.$ We assume that
 \[
 \forall\,z\in \{w:|w|\leq\rho\}:\quad g(z)^{-c}\det \Mtilde\neq 0,
 \]
 or, equivalently, $g(z)^{-c}\det {\cal M}=0\Leftrightarrow u_i(z)=u_j(z)$ for some $i\neq j.$
 \item We assume that
 \[
\lim_{z\to \rho}\left.\frac{\partial}{\partial u} \left( Q(z,u)W(z,u)-\sum_{l=1}^c Q(z,u_l(z))W(z,u_l(z))  \frac{\det {\cal M}(u;l)}{\det {\cal M} }    \right )     \right|_{u=\tau}\neq 0.
 \]

\end{enumerate}

\noindent
\textbf{Remark.} \textit{i)} The singular points of the equation $\Phi(z,u)=0$ are to be sought for among the zeroes of $\text{lc}(z)$ and of the resultant of $\Phi(z,u)$ and $\Phi^\prime(z,u)$ \cite{Hille62}. According to our assumptions, $\text{lc}(0)=0$ and $\rho$ is a zero of the resultant. Furthermore, 4 is satisfied if the resultant of $\Phi(z,u)$ and $Q(z,u)$ is non-zero in $|z|\leq\rho$ and $\text{lc}(z)$ is non-zero in $0<|z|\leq\rho$.\\
\textit{ii)} Assumption 6 is in particular fulfilled if the degrees of $Q(z,u)W(z,u)$ and the $t_i(z,u)$ as polynomials in $u$ are at most $c,$ since then at $z=\rho$ the term in brackets is a polynomial of degree $c$ with a \emph{simple} zero $u_1(\rho)=\tau,$ and $c-1$ not necessarily distinct zeroes $u_2(\rho),\ldots, u_c(\rho)\neq \tau.$\\
\textit{iii)} In the Meander case $\tau$ is the unique positive zero of $S^\prime (u),$ $S^{\prime\prime}(u)>0$ for $u>0,$ and $\rho =1/S(\tau).$ Separation of the small and large branches follows from a domination property: For $|z|\leq \rho,$ $i,j\geq 2$ we have $|u_i(z)|<u_1(|z|)$ and $|v_j(z)|>v_1(|z|).$ This and the uniqueness of $\rho$ follow in turn from the additional assumption of \emph{aperiodicity} of the step set, see \cite{BanFla02} for details. Since $\text{lc}(z)=s_dz,$ all branches remain finite for $z\neq 0.$ Furthermore, $Q(z,u)=u^c=Q(z,u)W(z,u)$ and the degree of $t_i(z,u)$ equals $i.$ Finally, $g(z)=s_{-c}z,$ and $\det {\cal M}=(s_{-c}z)^c\Delta_c(u_1,\ldots,u_c),$ i.e. $\det \Mtilde=(s_{-c}z)^c.$ \\

%

\medskip
\noindent
We now state the general version of Theorem \ref{theo1}, where the conditions in items 1, 2 and 3 generalise the drift being negative, zero and positive, respectively.
\begin{theorem}\label{generaltheo}
Let the sequences of $(Z_m),$ $(X_m)$ and $(H_m)$ of discrete random variables be defined in terms of formal power series $F(z,q,u)$ and $G_0(z,q)$ with non-negative coefficients as in \eqref{eq:XmZm} and \eqref{eq:Hm}, and let $F(z,q,u)\in \mathbb{C}[q,u][[z]]$ and $G_0(z,q),\ldots,G_{c-1}(z,q)\in \mathbb{C}[q][[z]]$ be completely determined by a functional equation \eqref{eq:functionalequation}, such that the above assumptions \ref{sec:analyticassumptions} are fulfilled.
\begin{enumerate}
\item If $\Phi(z,1)\neq 0$ in $\{|z|\leq \rho\},$ then \eqref{eq:negdrift} holds.
\item If $\tau=1$ (i.e. $\Phi(\rho,1)= 0$) and $\Phi(z,1)\neq 0$ in $\{|z|\leq \rho\}\setminus\{\rho\},$ then \eqref{eq:zerodrift} holds.
\item If there is $0<z_0<\rho,$ $\Phi(z_0,1) = 0$ and $\Phi(z,1)\neq 0$ for $\{|z|\leq z_0\}\setminus \{z_0\},$ then \eqref{eq:posdrift} holds with 
\[
\gamma= \frac{   \left. \frac{\partial}{\partial u} S(z_0,u)   \right |_{u=1}    }{S(z_0,1)+ z_0\left. \frac{\partial}{\partial z} S(z,1)   \right |_{z=z_0}      } .
\]
\item Equation \eqref{eq:excursionlaw} holds under either of the three conditions.

\end{enumerate}

\end{theorem}

\begin{lemma}
There is $\sigma>\rho$ such that the functions $G_k^{(n)}(z)$ are analytic in a slit disc 
\[
D(\rho,\sigma)=\left\{z\in\mathbb{C}\,|\, |z|<\sigma \right\}\setminus [\rho,\sigma)
\]  
and $F^{n,t}(u)=F^{n,t}(z,u)$ is analytic in ${\cal H}=D(\rho,\sigma)\times \mathbb{C}\setminus \{(z,v_j(z)) | z\in D(\rho,\sigma),\; j=1,\ldots,d \}.$
\end{lemma}
\begin{proof}
By the above separation property we can choose a $\sigma$ such that the $u_i(z)$ are separated from the $v_j(z),$ $i,j\ge 2,$ and $z\in D(\rho,\sigma).$
By assumption 4. and general theory \cite{Hille62} a singular point $z_0\in D(\rho,\sigma),$ $z_0\neq 0,$ for a small branch of the kernel equation can only be a finite branch point, i.e. two or more branches take the same finite value. Analytic continuation along a small circle with centre $z_0$ permutes those branches coalescing there. 

Now we show that with $y(z,u)$ analytic in ${\cal H}$ the function 
\[
T(z,u)=\frac{ \det{\cal N}(z,u;y) }{\Phi(z,u)\det {\cal M}}
\]
is analytic on ${\cal H}.$ To that end we first argue that $T(z,u)$ assumes only finite values. Define the row vector 
\[
R(z,w)=\left(y(z,w),t_0(z,w),\ldots,t_{c-1}(z,w)    \right),
\]
so the rows of ${\cal N}(z,u;y)$ are obtained by evaluating at $w=u,u_1,\ldots, u_c.$ 
Let $\widetilde{{\cal N}}(z,u;y)$ be the matrix obtained by replacement of the row $R(z,u_j(z))$ by the divided difference $R(z,\cdot)\left[u,u_1,\ldots,u_j \right],$ $j=1,\ldots,c.$ Since $R$ is analytic, these take a finite value at $(z,u)=(z_0,u_0),$ wether or not some of the branches coalesce in $(z_0,u_0),$ by the continuity of the divided differences at multiple nodes. Then 
\begin{equation}\label{eq:detNtilde}
T(z,u)=\frac{ \det \widetilde{ {\cal N} }(z,u;y) }{\text{lc}(z)\prod_{j=1}^d(u-v_j) \det \Mtilde},
\end{equation}
since $\det {\cal N}(z,u;y)=\Delta_{c+1} (u,u_1,\ldots,u_c)\det \widetilde{{\cal N}}(z,u;y).$ We hence have continuity at $(z_0,u_0).$

By Hartogs' Theorem \cite{Scheidemann05}, analyticity of $T(z,u)$ at $(z_0,u_0)$ follows from the analyticity of the partial functions 
\[
z\mapsto T(z,u_0),\quad u\mapsto T(z_0,u).
\]
The latter is a meromorphic function in $u,$ which has at most a removable singularity at $u_0.$ For the analyticity of the former observe that the expression for $T$ is a symmetric function in $u_1,\ldots,u_c.$ Analytic continuation of $T(z,u_0)$ along a small circle with centre $z_0$ permutes the branches $u_i$ and takes $T(z,u_0)$ to itself. With Morera's theorem analyticity follows \cite{Hille62}.

To prove the Lemma, one proceeds inductively. With $y(z,u)=Q(z,u)W(z,u),$ we immediately obtain the assertions for the $F^{0,t}(u),$ $t\ge0.$ For $G_k^{(0)}(z)$ we can argue similarly with their respective determinantal representations. Now let the assertions be true for all $(n-1,t),$ with $n\ge 1.$ The determinantal representations of $G_k^{(n)}(z)$ and $F^{n,0}(u),$ involve by \eqref{eq:derfcteq} functions $G_k^{(l)}(z)$ and $F^{l,t}(u),$ $l\leq n-1,$ $l+t\leq n,$ which are analytic by the induction hypothesis. With the above reasoning the proof is complete.
\end{proof}

\noindent
The qualitative singular behaviour at $z=\rho$ is treated in the following lemma. 

\begin{lemma}\label{lem:expansions}
\textit{i)} The functions $G^{(n)}_k(z),$ $k=0,\ldots,c-1,$ allow a Puiseux expansion in ascending powers of $\sqrt{1-z/\rho}.$ The same is true for $F^{n,t}(u_i(z))$ if $u_i(z)$ is analytic at $z=\rho.$\\
\textit{ii)} Assume that for an index set $L\subset\{2,\ldots, c\},$ the branches $u_{l}(z),\;l\in L$ form a cycle of length $k$ at $z=\rho,$ i.e. are conjugates of each other. Then $F^{n,t}(u_l(z)),$ $l\in L,$ allows a Puiseux expansion in ascending powers of $(1-z/\rho)^{1/2p},$ if $p$ is odd, and in ascending powers of $(1-z/\rho)^{1/p},$ if $p$ is even. 
\end{lemma}
\begin{proof}
By the closure properties of algebraic functions and the algebraicity of the $u_i,$ the functions in question are algebraic as rational functions in $u_1,\ldots,u_c.$ We can write $G^{(n)}_k(z)=P(u_1,\ldots,u_c)$ and $F^{n,t}(u)=Q(u,u_1,\ldots,u_c,$ where $P$ and $Q$ are symmetric in $u_1,\ldots,u_c.$ Now consider a small circle $C$ with centre $\rho.$ If we start at a $z_0\in C$ and continue $G^{(n)}_k(z)$ resp. $F^{n,t}(u_i)$ analytically describing $C$ once counterclockwise, $G^{(n)}_k(z_0)=P(u_1(z_0),\ldots,u_c(z_0))$ is taken to $P(v_1(z_0),u_2(z_0),\ldots,u_c(z_0))$ and, if $u_i$ is analytic at $\rho,$ $F^{n,t}(u_i)$ is taken to \linebreak $Q(u_i(z_0),v_1(z_0),u_2(z_0),\ldots,u_c(z_0)),$ since $u_1$ is taken to $v_1,$ $u_i$ remains fixed and the remaining branches are permuted. Describing $C$ once more , we end up at $G^{(n)}_k(z)$ resp. $F^{n,t}(u_i)$ again. By Puiseux' Theorem \cite{Hille62}, assertion \textit{i)} follows. Let $\{w_1,\ldots,w_p\}=\{u_l,\;l\in L\}$ in such a way, that describing $C$ once counterclockwise takes $w_j(z)$ into $w_{j+1}(z)$ (indices mod $p$). Then describing $C$ once takes $F^{n,t}(w_j(z_0))$ to $Q(w_{j+1}(z_0),v_1(z_0),u_2(z_0),\ldots,u_c(z_0)).$ Notice that, for odd $p,$ repeated analytic continuation of the pair $(u_1,w_1)$ along $C$ yields the cycle of length $2p $
\[
\begin{split}
(u_1,w_1)\longrightarrow (v_1,w_2)\longrightarrow (u_1,w_3)\longrightarrow \ldots\longrightarrow (v_1,w_{k-1}) \longrightarrow (u_1,w_k)\\ \longrightarrow (v_1,w_1)\longrightarrow (u_1,w_2)\longrightarrow \ldots\longrightarrow (v_1,w_k)\longrightarrow (u_1,w_1),     
\end{split}
\]
hence again by Puiseux' Theorem, the first assertion in \textit{ii)} follows. In a similar fashion, for even $p$ we have two disjoint cycles of length $k,$ one containing all the $(u_1,w_{2j})$ and one containing all the $(u_1,w_{2j+1}).$ This proves the second assertion of \textit{ii)}. 
\end{proof}

\noindent
By Tailor's formula the Puiseux expansion of $S^\prime(z,u_1(z))$ at $z=\rho$ starts with
\begin{equation}\label{eq:Sprime}
\begin{split}
S^\prime(z,u_1(z))&=S^\prime\left(z,\tau-\beta\sqrt{1-z/\rho}+O(1-z/\rho) \right)\\
&=-S^{\prime\prime}(\rho,\tau)\beta\sqrt{1-z/\rho}+O(1-z/\rho)\quad(z\longrightarrow \rho).
\end{split}
\end{equation}
Let similarly the small branches $w_1(z),\ldots,w_p(z)$ form a cycle at $z=\rho,$ with a local expansion
\begin{equation}\label{eq:wj}
 w_j(z)=\alpha + \gamma \xi^{(j-1)} (1-z/\rho)^{1/p}+\ldots \quad (z\longrightarrow \rho),
\end{equation}
where $\xi=\exp(2\pi i /p)$ is a $p$th root of unity.
Then we have 
\[
\Phi(\rho,\alpha)=\Phi^\prime(\rho,\alpha)=\ldots =\Phi^{(p-1)}(\rho,\alpha)=0, \quad \Phi^{(p)}(\rho,\alpha)\neq 0,
\]
and the Puiseux expansion of $\Phi^{(l)}(z,w_j(z))$ starts with
\begin{equation}\label{eq:Phiwj}
\Phi^{(l)}(z,w_j(z))=\frac{1}{(p-l)!}\Phi^{(p)}(\rho,\alpha) \gamma^{p-l}\xi^{(j-1)(p-l)} (1-z/\rho)^{(p-l)/p}+\ldots
\end{equation}

The functions $F^{n,t}(u_1(z)),$ $F^{n,t}(1)$ and $G_k^{(n)}$ are easily seen to be rational functions in the $u_i,$ and by the discussion above they admit a Laurent series expansion in powers of $\sqrt{1-z/\rho}.$ We compute the leading singular term of these series, starting with $F^{n,t}(u_1(z)).$
\begin{lemma}\label{lem:singFntu1}
The function $F^{n,t}(z,u)$ evaluated at the largest of the small branches $u_1(z)$ has the following leading singular behaviour at $z=\rho:$
\begin{equation}\label{eq:singFnt}
F^{n,t}(z,u_1(z))=\frac{a_{n,t}}{(1-z/\rho)^{3n/2+t/2+1/2}}+O\left( \frac{1}{(1-z/\rho)^{3n/2+t/2}} \right)\quad (z\longrightarrow \rho),
\end{equation}
where the numbers $a_{n,t}$ satisfy the recursion
\begin{equation}\label{eq:recursionant}
a_{n,t}=\frac{t}{2\beta}\cdot a_{n,t-1} +\frac{\beta}{2}\cdot \frac{n}{t+1}\cdot a_{n-1,t+2}
\end{equation}
with 
\begin{equation}\label{eq:a00}
a_{0,0}=\lim_{z\longrightarrow \rho}\frac{\beta \left.\frac{\partial}{\partial u}\left( \det  {\cal N}(z,u;Q\cdot W)\right)\right |_{u=u_1}}{2Q(z,u_1(z))\det {\cal M}}\neq 0
\end{equation}
In particular, this limit exists. If the small branches $w_1(z),\ldots,w_p(z)$ form a cycle of length $p$ at $z=\rho$ with an expansion \eqref{eq:wj}, then we have
\begin{equation}
F^{n,t}(z,w_j(z))=O\left( \frac{1}{(1-z/\rho)^{3n/2-1/2+t/p}} \right)\quad (z\longrightarrow \rho).
\end{equation}
If $u_i(z),$ $i\geq 2,$ is regular at $z=\rho,$ then
\begin{equation}
F^{n,t}(z,u_i(z))=O\left( \frac{1}{(1-z/\rho)^{3n/2-1/2}} \right)\quad (z\longrightarrow \rho).
\end{equation}
Furthermore, for $n\geq 1$ the singular behaviour of the functions $G_k^{(n)}(z)$ is 
\begin{equation}\label{eq:singGkn}
G_k^{(n)}(z)=O\left(\frac{1}{(1-z/\rho)^{3n/2-1/2}} \right)\quad (z\longrightarrow \rho).
\end{equation}
\end{lemma}
\begin{proof}
If we consider the $F^{n,t}$ organised in an array ($n$ indexing rows, $t$ columns), we can compute the entries inductively, row by row, where each skipping to the next row requires an application of the kernel method. We prove the lemma accordingly.

First of all, $a_{0,0}$ is well defined by the continuity of the defining term at $z=\rho,$ since a factor $g(z)^c\Delta(u_1,\ldots,u_c)$ can be cancelled in the numerator and the denominator, and by Assumption 6 it is non-zero.

\noindent
\textbf{Assertions for $F^{0,0}(u_1)$ and $F^{0,t}(u_i),$ $t\ge 0, i\ge 2$:} The assertion for $(n,t)=(0,0)$ follows from the explicit form \eqref{eq:generalsol} of  $F^{0,0}(u),$ where $y(z,u)=Q(z,u)W(z,u).$ We have
 \begin{equation}
 F^{0,0}(u_1(z))=\frac{\left.\frac{\partial}{\partial u}\left( \det  {\cal N}(z,u;Q\cdot W)\right)\right |_{u=u_1(z)}}{-zQ(z,u_1(z)S^{\prime}(z,u_i(z)))\det {\cal M}}\sim \frac{a_{0,0}}{\sqrt{1-z/\rho}}\quad (z\longrightarrow \rho),
\end{equation}
where we used the fact that for a solution $u_i$ of the kernel equation we have $\Phi^{\prime}(z,u_i(z))=-zQ(z,u_i(z))S^{\prime}(z,u_i(z)).$ Furthermore we used the expansion \eqref{eq:Sprime} and the defining relations $\rho=1/S(\rho,\tau)$ and $\beta=\sqrt{ 2S(\rho,\tau)/S^{\prime\prime}(\rho,\tau) }.$ Moreover, from \eqref{eq:xk} with $y(z,u)=Q(z,u)W(z,u)$ we see that $G_k^{(0)}$ is finite at $z=\rho.$ By the previous lemma it has at most a square root singularity at $\rho.$

As for $F^{0,t}(u_j(z)),$ $j\ge 2,t>0$ we look at a representation \eqref{eq:detNtilde} (and derivatives w.r.t. $u$ thereof) to see that evaluation at $u=u_j(z)$ leads to a finite value at $z=\rho$ for every $t\ge 0.$ By the separation property and together with the previous lemma we obtain the desired bound. 

\noindent 
\textbf{Induction step for $u_1(z)$ and branches $u_l(z)$ analytic at $z=\rho$:} Let $n\ge 0$ and assume that the estimate for $G^{(n)}_k$ and all $F^{r,s}(u_1(z))$ (resp. $F^{r,s}(u_1(z))$) are established, where $r\leq n-1.$ We can compute the estimate for $F^{n,t}(u_1(z))$ for $t\geq 0$ from the $t+1$st derivative of \eqref{eq:derfcteq} which can be written as (with the obvious adjustments in the case $n=0$)
\begin{equation}\label{eq:Fnt}
\begin{split}
(1-zS(z,u))F^{n,t+1}(u)&-(t+1)zS^{\prime}(z,u)F^{n,t}(u)=\\
&{t+1 \choose 2}zS^{\prime\prime}(z,u)F^{n,t-1}(u)+zS(z,u)nF^{n-1,t+2}(u)  \\
&+R_{n,t+1}(u),
\end{split}
\end{equation}
where $R_{n,t+1}(u)$ contains all the remaining derivatives of total order $\leq n+t+1$ and of order $<n$ in $q.$ For $u=u_1(z),$ we see by induction hypothesis that 
\[
R_{n,t+1}(u_1(z))=O\left( (1-z/\rho)^{-(3n/2+t/2-1/2)} \right)
\]
and is hence of neglegible order, as $z\longrightarrow \rho.$ With the the above Tailor expansion \eqref{eq:Sprime} of $S^\prime(u_1(z))$ equation \eqref{eq:Fnt} reads on the level of leading order terms
\[
\begin{split}
F^{n,t}(u_1(z))=\frac{t}{2\beta\sqrt{1-z/\rho}}\cdot&\frac{a_{n,t-1}}{(1-z/\rho)^{3n/2+t/2}}\\+\frac{n}{t+1}\cdot
\frac{S(\rho,\tau)}{S^{\prime\prime}(\rho,\tau)\beta\sqrt{1-z/\rho}}\cdot&\frac{a_{n-1,t+2}}{(1-z/\rho)^{3n/2+t/2}}+O\left( \frac{1}{(1-z/\rho)^{3n/2+t/2}} \right)\quad (z\longrightarrow \rho).
\end{split}
\]
Recalling the definition of $\beta$ this establishes the recursion equation. Notice that this settles in particular the assertions for $F^{0,t}(u_1(z).$

One argues similarly for the branches $u_l,$ where we first see that 
\[
R_{n,0}(u_l(z))=O\left( (1-z/\rho)^{-(3n/2-1/2)} \right)
\] 
(with the $G_k^{(n)}$ the dominating terms) and since $S^\prime(z,u_l(z))\longrightarrow \epsilon\neq 0,$ we have the same estimate for $F^{n,0}(u_l(z)).$ The estimate for $F^{n,t}(u_l(z))$ follows by induction on $t.$

\noindent
\textbf{Induction step $n-1\rightarrow n$ for $G_k^{(n)}$:} 
Now let $n\geq 1$ and assume the assertions to be true for all $(k,t)$ with $k \leq n-1.$ We apply now the kernel method to compute the $G^{(n)}_k$, i.e. we substitute $u_1,\ldots,u_c$ into \eqref{eq:derfcteq} (multiplied by $Q(z,u)$) and obtain $c$ linear equations for $G_k^{(n)},$ $k=0,\ldots,c-1,$ the solution to which is given by \eqref{eq:xk}, when $y(z,u)$ is replaced by $\widetilde{RHS}_n(z,u)=Q(z,u)RHS_n(z,u).$  

\begin{equation*}
G^{(n)}_k=\frac{(-1)^{k+c-1} }{ s_{-c}z }\sum_{l=1}^c\widetilde{RHS}_n(u_l)\frac{\phi_k(u_1,\ldots,u_{l-1},u_{l+1},\ldots,u_c)}{\det \Mtilde}
\prod_ {\begin{smallmatrix}  j=1\\ j\neq l \end{smallmatrix}    } ^c  \frac{1}{u_l-u_j}, 
\end{equation*}
By induction hypothesis, we have the estimates
\begin{equation}\label{eq:RHS}
\widetilde{RHS}_n(z,u_i(z))=O\left(\frac{1}{(1-z/\rho)^{3n/2-2+\delta(i)}}\right),
\end{equation}
where $\delta(1)=3/2,$ $\delta(i)=0,$ if $u_i(z)$ is analytic at $z=\rho,$ and $\delta(i)=1/p,$ if $u_i(z)=w_j(z)$ as in \eqref{eq:wj}. Therefore, each summand is $O\left((1-z/\rho)^{-(3n/2-1/2)}\right),$ since the products in the denominator is tends to a constant $\neq 0$ for $u_1$ and the branches $u_l$ analytic at $z=\rho,$ while for a branch $w_j$ as in \eqref{eq:wj} it tends to $0$ like $(1-z/\rho)^{(p-1)/p}.$ The estimate for $G^{(n)}_k$ follows.

\noindent
\textbf{Induction step for coalescing small branches $w_j$:} Assume that $n\ge 1$ and the assertions are true for all pairs $(r,s),$ where $r\leq n-1.$ 
According to the representation \eqref{eq:generalsol}, we have
\begin{equation}
\begin{split}
\Phi(z,u)&F^{n,0}(u)=\\
&\widetilde{RHS}_n(z,u)-\sum_{l=1}^c \widetilde{RHS}_n(z,u_l(z))\frac{\det{\Mtilde}(u;l) }{\det {\Mtilde}}   \prod_ {\begin{smallmatrix}  j=1\\ j\neq l \end{smallmatrix}    } ^c  \frac{u-u_j}{u_l-u_j}.
\end{split}
\end{equation}
The estimate for $F^{n,t}(w_j)$ is computed by induction on $t$ by taking $t+1$ derivatives of this equation and evaluating at $w_j.$ 
\begin{equation}
\begin{split}
&(t+1)  F^{n,t}(w_j(z))=\\
&-\sum_{k=2}^{t+1} {t+1 \choose k}\frac{\Phi^{(k)}(z,w_j)}{ \Phi^\prime(z,w_j(z)) }F^{n,t+1-k}(w_j(z)) +\frac{\widetilde{RHS}^{(t+1)}_n(z,w_j(z))}{\Phi^\prime(z,w_j(z))} \\
&-\frac{1}{\Phi^\prime(z,w_j(z))}\sum_{l=1}^c \widetilde{RHS}_n(z,u_l(z))\left(\frac{\partial}{\partial u} \right)^{t+1}\frac{\det{\Mtilde}(u;l) }{\det {\Mtilde}}   \prod_ {\begin{smallmatrix}  j=1\\ j\neq l \end{smallmatrix}    } ^c \left. \frac{u-u_j}{u_l-u_j} \right |_{u=w_j}.
\end{split}
\end{equation}
Observe that for $z\longrightarrow \rho$ and by \eqref{eq:wj} we have
\[
 \left(\frac{\partial}{\partial u} \right)^{r} \prod_ {\begin{smallmatrix}  j=1\\ j\neq l \end{smallmatrix}    } ^c \left.  \frac{u-u_j}{u_l-u_j}   \right |_{u=w_j} = \begin{cases} O\left((1-z/\rho)^{-r/p}   \right),\; u_l\in \{w_1,\ldots,w_p\},\;r\leq p-1,\\
 O\left((1-z/\rho)^{(p-1)/p}\right ), u_l\in \{w_1,\ldots,w_p\},\;r\geq p,\\
 O\left((1-z/\rho)^{(p-r)/p}\right ), u_l\notin \{w_1,\ldots,w_p\},\;r\leq p-1,\\
 O\left(1 \right ), u_l\notin \{w_1,\ldots,w_p\},\;r\geq p.\\
 \end{cases}
\]
Furthermore by the expansions \eqref{eq:Phiwj} we see that 
\[
\frac{\Phi^{(k)}(z,w_j)}{ \Phi^\prime(z,w_j(z)) }=\begin{cases}O\left((1-z/\rho)^{-(k-1)/p}\right ),\;k\le p-1 \\
O\left((1-z/\rho)^{-(p-1)/p}\right ),\;\text{otherwise.}
\end{cases}
\]
By induction hypothesis, $\widetilde{RHS}_n^{(t+1)}(z,w_j)=O\left((1-z/\rho)^{-(3n/2-2+(t+1)/p)}\right ).$ With the above estimates \eqref{eq:RHS} for $\widetilde{RHS}_n(u_i),$ $i=1,\ldots, c,$ we can prove the assertions for $F^{n,t}(w_j(z))$ by induction on $t,$ by keeping track of the singular orders at $z=\rho.$
\end{proof}
\section{Limit distribution of $X_m,$ excursion area}\label{sec:excarea}

Before we apply the above result to prove the limit laws via coefficient asymptotics, we recall the Transfer Theorem \cite{FlaOdl90}. 
\begin{theorem}[Theorem VI.3 in \cite{FlaSed09}]\label{app:transfer}
Let $\alpha\in\mathbb{C}\setminus \{0,-1,-2,\ldots \}$ and $F(z)$ be analytic in the open indented disc
\[
D(\rho,\sigma,\phi)=\left\{z\in\mathbb{C}\,|\, |z|<\sigma,\;z\neq \rho,\;|\arg(z-\rho)|< \phi \right\}
\]  
where $0<\rho<\sigma$ and $0<\phi<\pi/2.$ If in the intersection of a small neighbourhood of $\rho$ with $D(\rho,\sigma,\phi)$ $F(z)$ satisfies the condition
\[
F(z)\sim(1-z/\rho)^{-\alpha}\quad (z \longrightarrow \rho),
\]
then $\left[z^n\right]F(z)\sim \rho^{-n}\dfrac{n^{\alpha-1}}{\Gamma(\alpha)},$ for $n\longrightarrow \infty.$ \qed
\end{theorem}

\noindent
With $Y(z,u)=W(z,u)Q(z,u)$ we have by \eqref{eq:Gk}
\[
G^{(0)}_k(z)=\frac{\det {\cal G}_k(Y) }{\det {\cal M}  } .
\]
By lemmas \ref{lem:expansions} and \ref{lem:singFntu1} $G^{(0)}_k(z)$ has an expansion in non-negative powers of $\sqrt{1-z/\rho}.$ In order to access the coefficient of $\sqrt{1-z/\rho}$ we regard $G^{(0)}_k$ formally as a function of the variable $u_1$ into which we substitute the expansion $u_1(z)=\tau-\beta \sqrt{1-z/\rho}+O(1-z/\rho).$ By Tailor's theorem we hence have
\[
G^{(0)}_k(z)=G^{(0)}_k(\rho)-\varepsilon \sqrt{1-z/\rho} +O(1-z/\rho)
\]
where
\[
\begin{split}
\varepsilon&=\lim_{z\to\rho}\frac{\det {\cal M} \left.\frac{\partial}{\partial u_1}\det {\cal G}_k\right |_{u_1=\tau} -\det {\cal G}_k\left. \frac{\partial}{\partial u_1}\det {\cal M} \right |_{u_1=\tau} }{\det {\cal M}^2}\\
&=\lim_{z\to \rho}\frac{\det {\cal M}_{k,1}\left.\frac{\partial}{\partial u} \det  {\cal N}(z,u;y(z,u))\right |_{u=u_1} }{\det {\cal M}^2 }\\
&=2 a_{0,0}Q(\rho,\tau) \lim_{z\to \rho} \frac{\det {\cal M}_{k,1} }{ \det {\cal M} }, 
\end{split}
\]
Here we used the relation \eqref{eq:Plucker1} and the definition \eqref{eq:a00} of $a_{0,0}.$ Furthermore, recalling that \linebreak$zQ(z,u_1(z))S(z,u_1(z)=Q(z,u_1(z)),$ we have by Lemma \ref{lem:singFntu1} that
\begin{equation}\label{eq:RHSestimate}
RHS_n(z,u_1(z))=Q(z,u_1(z))nF^{n-1,1}(u_1(z))+O\left((1-z/\rho)^{-(3n/2-1)}\right)\quad (z \longrightarrow \rho),
\end{equation}
which dominates $RHS_n(u_i(z)),$ $i\ge 2.$ Therefore, with \eqref{eq:Gk} and the leading order term given in \ref{lem:singFntu1}
 we have for $z\longrightarrow \rho$
\begin{equation}\label{eq:singG0n}
\begin{split}
\frac{1}{n!}G_0^{(n)}(z)&=\frac{1}{n!} Q(z,u_1(z))nF^{n-1,1}(u_1(z)) \frac{\det {\cal M}_{k,1} }{ \det {\cal M} } +O\left((1-z/\rho)^{-(3n/2-1)}\right)\\
& \sim       \frac{1}{(1-z/\rho)^{3n/2-1/2}}          \frac{a_{n-1,1}}{(n-1)!}  Q(\rho,\tau)\lim_{z\to \rho} \frac{\det {\cal M}_{k,1} }{ \det {\cal M} }\quad (z \longrightarrow \rho).
\end{split}
\end{equation}
For the random variables $X_m$ defined as in \eqref{eq:XmZm} we hence have by the Transfer Theorem the following asymptotics for the factorial moments,
\begin{equation}\label{eq:asExMo}
\frac{1}{n!}\mathbb{E}\left(\left(X_m\right)_n\right)=\frac{1}{n!}\frac{\left[z^m \right]G_0^{(n)}(z) }{\left[z^m \right]G^{(0)}_0(z)}\sim 
 \frac{-a_{n-1,1}}{2a_{0,0}(n-1)! } \frac{\Gamma(-1/2)}{\Gamma(3n/2-1/2)}m^{3n/2}
\end{equation}
This shows that the ordinary and factorial moments are asymptotically equal and hence the moment convergence for the sequence $\frac{\beta}{\sqrt{2}} m^{-3/2} X_m.$ By the recursion equation \eqref{eq:recursionant}, the normalised sequence $C_{n,t}=\frac{2^{3n+t}\beta^{n+t}a_{n,t}}{n!t!a_{0,0}}$ is precisely the one from part \textit{i)} of Theorem \ref{theo:BrownianMeander}. We hence have the moment convergence
\[
\lim_{m\rightarrow\infty}\frac{1}{n!}\mathbb{E}\left(\left( \frac{\beta X_m}{\sqrt{2}m^{3/2}} \right)^n\right)= \frac{-\Gamma(-1/2)  }{\Gamma(3n/2-1/2)} 2^{ -\frac{7n-2}{2} } C_{n-1,1},
\]
which implies the asserted weak convergence to Brownian excursion area \cite{Chung74}, once Theorem \ref{theo:BrownianMeander} is established. This is done in Section \ref{sec:ProofLemmas}. Our proof also shows that the same limit distribution arises, if we choose another $G_k$ for the definition of $X_m,$ i.e. the distribution of area under meanders ending at a fixed altitude $<c$ is asymptotically ${\cal BEA}.$ \qed
\section{Limit distribution of $Z_m:$ negative drift}\label{sec:negdrift}
We now consider the case when $\Phi(z,1)\neq 0$ for $|z|\leq\rho,$ which in the frame work of discrete meanders simply means that the drift $\gamma=S^\prime(1)/S(1)$ is negative (and hence $\tau>1$). It turns out, that the dominant role is played by the quantities $F^{n,t}(u_1)$.%
\begin{lemma}\label{Fntnegdrift} 
We have 
\[
F^{n,t}(1)=o\left(F^{n,t}(u_1(z))\right)\quad(z\longrightarrow \rho). 
\]
\end{lemma}
\begin{proof}
Once the assertion is true for $(n,0),$ we can prove the statement for $(n,t),$ $t\geq 1$ by induction on $t$ as follows. Assume the statement to be true for all $k < t$ and consider the $t$th derivative of \eqref{eq:derfcteq} w.r.t. $u$ evaluated at $u=1.$ By induction hypothesis, every (known) term is dominated by $F^{n,t-1}(u_1(z))$ for $z \longrightarrow \rho.$ Since $\Phi(z,1)\neq 0$ for  $|z| \leq \rho,$ solving for $F^{n,t}(1)$ shows that it is also at most of order as $F^{n,t-1}(u_1(z))$ which is in turn $o\left(F^{n,t}(u_1(z))\right)$ by Lemma \ref{lem:singFntu1}. 

The assertion for $(n,t)=(0,0)$ follows from the explicit representation \eqref{eq:generalsol} of $F^{0,0}(u)$ and Lemma \ref{lem:singFntu1}. Now assume that the assertion is true for all $(k,t)$ with $k \leq n-1.$ By the representation of $F^{n,0}(u)$ obtained applying \eqref{eq:generalsol} with $y(z,u)=RHS_n(z,u),$ we see that $F^{n,0}(1)$ is $O\left(F^{n-1,1}(u_1(z))\right)$ as $z \longrightarrow \rho,$ which in turn is $o\left(F^{n,0}(u_1(z))\right).$
\end{proof}
Let again $Y(z,u)=Q(z,u)W(z,u).$ We see easily that the singular behaviour of $F^{0,0}(z,1) =
 \det {\cal N}(z,1;Y(z,u))  / ( \Phi(z,1) \det{\cal M} ) $ is of square root type. Similarly to Section \ref{sec:excarea}, we access the coefficient of $\sqrt{1-z/\rho}$ by regarding $F^{0,0}(z,1)$ formally as a function of the variable $u_1$ into which we substitute the expansion $u_1(z)=\tau-\beta \sqrt{1-z/\rho}+O(1-z/\rho).$ By Tailor's theorem we hence have
 \[
 F^{0,0}(z,1)\sim \left. F^{0,0}(\rho,1)\right |_{u_1=\tau}-\eta\beta \sqrt{1-z/\rho}\quad (z\longrightarrow \rho),
 \]
where $\eta$ is given by
\[\begin{split}
-\eta&=-\lim_{z\to \rho}\left.\frac{\partial}{\partial u_1} F^{0,0}(1)\right |_{u_1=\tau}\\
&=\lim_{z\to \rho}\frac{\det  {\cal N}(z,1;Y)\left.\frac{\partial}{\partial u_1}\det {\cal M}\right |_{u_1=\tau}- \det {\cal M}\left.\frac{\partial}{\partial u_1}\det {\cal N}(z,1;Y)\right |_{u_1=\tau} }{ \Phi(z,1) \det {\cal M}^2 }\\
&=\lim_{z\to \rho}\frac{\det {\cal M}(1;1) \left.\frac{\partial}{\partial u} \det  {\cal N}(z,u;Y)\right |_{u=\tau} }{ \Phi(z,1)\det {\cal M}^2},
\end{split}
\]
where we used the relation \eqref{eq:Plucker2}. Finally, we plug in definition \eqref{eq:a00} and get 
\[
-\eta=\frac{2a_{0,0}Q(\rho,\tau)}{\beta\Phi(\rho,1)}  \lim_{z\to \rho}\frac{\det {\cal M}(1;1)}{\det {\cal M}}.
\]
For the leading singular behaviour of $F^{n,0}(1)$ we recall \eqref{eq:RHSestimate} and see by virtue of Lemma \ref{Fntnegdrift} that $RHS_n(z,1)=o(RHS_n(z,u_1(z))).$ With the representation \eqref{eq:Gk}  
\[
\begin{split}
\frac{F^{n,0}(1)}{n!}&=\frac{-1}{\Phi(z,1)}\cdot\frac{nQ(z,u_1(z))}{n!}F^{n-1,1}(u_1(z))\frac{\det{\cal M }(1;1)}{\det{\cal M}}
+ O\left( (1-z/\rho)^{-(3n/2-1)} \right) \\
&\sim \frac{-1}{(n-1)!} \frac{a_{n-1,1}}{(1-z/\rho)^{3n/2-1/2}} \frac{Q(\rho,\tau)}{\Phi(\rho,1)}\lim_{z\to \rho}\frac{\det {\cal M}(1;1)}{\det {\cal M}} \quad (z\longrightarrow \rho).
\end{split}
\]
The moments of $Z_m$ are seen to be asymptoticaly equal to those in \eqref{eq:asExMo}, namely
\begin{equation*}
\frac{1}{n!}\mathbb{E}\left(\left(Z_m\right)_n\right)=\frac{1}{n!}\frac{\left[z^m \right]F^{n,0}(z,1) }{\left[z^m \right]F^{0,0}(z,1)}\sim 
 \frac{-a_{n-1,1}}{2a_{0,0}(n-1)! } \frac{\Gamma(-1/2)}{\Gamma(3n/2-1/2)}m^{3n/2}.
\end{equation*}
The rest of the argument is as in the previous section.\qed 

\section{Limit distribution of $Z_m:$ zero drift}

As mentioned in the remark following Theorem \ref{theo1}, the convergence of the joint distribution of meander area and endpoint is a simple consequence of the weak convergence result in \cite{Iglehart74}. The formula \eqref{eq:asMeMo} for the joint moments given below yields the recursion \eqref{eq:Qnt} in Theorem \ref{theo:BrownianMeander}. If in turn we consider Theorem \ref{theo:BrownianMeander} as given, then the weak convergence to $\left({\cal BMA},{\cal B}^{me}(1)\right)$ follows from the moment convergence \eqref{eq:asMeMo}, which is as a consequence of the functional equation.  
This is case 2 in Theorem \ref{generaltheo}, $\tau=1$ and $\Phi(z,1)\sim Q(\rho,1) (1-z/\rho)$ as $z\longrightarrow \rho.$ 
\begin{lemma}\label{lem:singFnt1}
The mixed derivatives evaluated at $q=u=1$ have the following leading singular behaviour as $z\longrightarrow \rho.$
\begin{equation}\label{eq:singFnt1}
F^{n,t}(z,1)=\frac{b_{n,t}}{(1-z/\rho)^{3n/2+t/2+1/2}}+O\left( \frac{1}{(1-z/\rho)^{3n/2+t/2}} \right)\quad (z\longrightarrow \rho),
\end{equation}
where the numbers $b_{n,t}$ for $(n,t)=(0,0)$ and $(n,t)=(0,1)$ are given by 
\begin{equation}
b_{0,0}=2a_{0,0} \text{ and } b_{0,1}=\frac{1}{\beta}b_{0,0},
\end{equation}
for $(n,t)$ with $t\geq 1$ by the recursion relation
\begin{equation}\label{eq:recbnt}
b_{n,t}=\frac{t(t-1)}{\beta^2}b_{n,t-2}+nb_{n-1,t+1},
\end{equation}
and for $(n,t)$ with $t=0$ and $n\geq 1$
\begin{equation}
b_{n,0}=nb_{n-1,1}+na_{n-1,1}.
\end{equation}
The numbers $a_{n,t}$ are defined in Lemma \ref{lem:singFntu1}.
\end{lemma}
\begin{proof}
The calculations are again similar to those in the proof of Lemma \ref{lem:singFntu1}. Let $Y(z,u)=Q(z,u) W(z,u).$ We first give the proofs for $(n,t)=(0,0)$ and $(0,1).$ As for $b_{0,0},$ we argue similarly as in the previous section that
\[
\frac{\det {\cal N}(z,1;Y) }{ \det{\cal M}  } = -\beta\mu \sqrt{1-z/\rho}+O(1-z/\rho)\quad (z\longrightarrow \rho),
\]
with 
\[
 \mu= \lim_{z\to\rho} \left.    \frac{ \partial }{ \partial u_1 }  \frac{ \det {\cal N}(z,1;Y)   }{  \det{\cal M}  }      \right |_{u_1=1} 
\]
Since $\Phi(z,1)\sim Q(z,1)(1-z\rho)$ we have
\[
\begin{split}
F^{0,0}(1) = \frac{\det {\cal N}(z,1;Y) }{ \Phi(z,1) \det{\cal M}     }&\sim \frac{\mu }{ \Phi(z,1)  }\cdot (-\beta) \sqrt{1-z/\rho}\\
&\sim -\frac{\mu}{ Q(\rho,1) }\cdot \frac{ \beta}{ \sqrt{1-z/\rho}}\quad (z\longrightarrow \rho).
\end{split}
\]
The formula for $b_{0,0}$ follows from the simple observation that
\[
-\left. \frac{\partial}{\partial u_1} \det {\cal N}(z,1;Y) \right |_{u_1=1} =\left. \frac{\partial}{\partial u} \det {\cal N}(z,u;Y)\right |_{u=u_1=1}. 
\]
For $b_{0,1}$ we consider
\[
F^{0,1}(1)=  \frac{\left. \frac{\partial}{\partial u} \det {\cal N}(z,u;Y) \right |_{u=1}}{ \Phi(z,1) \det{\cal M}   }
- \frac{ \det {\cal N}(z,1;Y) \Phi^{\prime}(z,1) }{ \Phi(z,1)^2 \det{\cal M}     }.
\]
As $z\longrightarrow \rho,$ the second summand is of order $O(1/\sqrt{1-z\rho})$ since $\det {\cal N}(z,1;Y)$ is of order $O(\sqrt{1-z/\rho})$ and $\Phi^\prime(z,1)$ of order $O(1-z/\rho).$ The first summand is precisely of order $1/(1-z/\rho)$ and hence  
\[
F^{0,1}(1)\sim  \frac{\left. \frac{\partial}{\partial u} \det {\cal N}(z,u;Y) \right |_{u=u_1=1}}{ Q(\rho,1) \det{\cal M}   }\cdot \frac{1}{1-z\rho}\sim\frac{b_{0,0}}{\beta}\cdot \frac{1}{1-z\rho}\quad (z\longrightarrow \rho).
\]

The remaining cases are shown by induction. For the $t\geq 0$ cases, the $t$th derivative of equation \eqref{eq:derfcteq} can be rewritten as
\begin{equation}\label{eq:Fnt1}
\begin{split}
(1-zS(u))F^{n,t}(u)=&tzS^{\prime}(z,u)F^{n,t-1}(u)\\
+&{t \choose 2}zS^{\prime\prime}(z,u)F^{n,t-2}(u)+zS(z,u)nF^{n-1,t+1}(u) +R_{n,t}(u),
\end{split}
\end{equation}
where we collected the $F^{k,l}(u)$ with $k+l<n+t$ or $k+l\leq n+t$ and $k\leq n-2$ and the terms involving $G_i^{(j)},$ $j\leq n$ in $R_{n,t}(u).$ 

We now consider the cases $t\geq 1.$ Upon setting $u=1,$ we see that $S^\prime(z,1)=O(1-z/\rho),$ and therefore the term $zS^\prime(z,1)F^{n,t-1}(1)$ is insignificant for leading order considerations.  By induction, $F^{n,t-2}(1)$  and $F^{n-1,t+1}(1)$ strictly dominate the terms in $R_{n,t}(1)$ as $z\longrightarrow \rho,$ and the recursion relation for $t\geq 1$ follows by comparing leading singular terms, recalling that $\rho =S(1)$ and $\beta^2 =2S(1)/S^{\prime\prime}(1).$ 

In order to prove the relation for $t=0,$ we have to combine \eqref{eq:generalsol} with \eqref{eq:derfcteq} again, leading to
\[
\begin{split}
F^{n,0}(u)=&\frac{1}{\Phi(z,u)} \left(  nzQ(z,u)S(z,u)F^{n-1,1}(u)-\sum_{l=1}^c nQ(z,u_lF^{n-1,1}(u_l)\frac{ {\cal M }(u;l)   }{ {\cal M }    }\right) \\
 +& \frac{1}{\Phi(z,u)} \left(  Q(z,u)R_{n,0}(u)-\sum_{l=1}^c Q(z,u_l)R_{n,0}(u_l) \frac{ {\cal M }(u;l)   }{ {\cal M }    }\right),
\end{split}
\]
where we used that $Q(z,u_l)=zQ(z,u_l)S(z,u_l).$ By induction hypothesis, for $u=1$ and for $z\longrightarrow \rho$ only the first line of the previous equation is of interest. By Lemma \ref{lem:singFntu1} the terms with $F^{n-1,1}(u_l)$ for $l\geq 2$ are also neglegible, which finally leads to
\[
F^{n,0}(1)\sim\frac{nzQ(z,1)S(z,1)F^{n-1,1}(1)- nQ(z,u_1)F^{n-1,1}(u_1)}{Q(z,1)(1-zS(z,1))} \sim n\frac{F^{n-1,1}(1)-F^{n-1,1}(u_1)}{(1-zS(\rho,1))}.
\]
The assertion follows from comparing leading singular terms.
\end{proof}
As above, an application of the Transfer Theorem gives the leading asymptotic behaviour of the joint factorial moments
\begin{equation}\label{eq:asMeMo}
\frac{1}{n!t!}\mathbb{E}\left(\left(Z_m\right)_n\left(H_m\right)_t\right)=\frac{1}{n!t!}\frac{\left[z^m \right]F^{n,t}(z) }{\left[z^m \right]F^{0,0}(z)}\sim \frac{b_{n,t}\Gamma(1/2)}{n!t!b_{0,0}\Gamma(3n/2+t/2+1/2)}m^{3n/2+t/2}.
\end{equation}
This implies that factorial and ordinary moments are asymptotically equal. If $n=0$, then 
\begin{equation}
\mathbb{E}\left((H_m)_t\right)=\frac{ \left[ z^m    \right]F^{0,t}(1)    }{ \left[ z^m   \right]F^{0,0}(1) }\sim \frac{b_{0,t}m^{t/2-1/2}\Gamma(1/2) }{ b_{0,0}m^{-1/2}\Gamma(t/2+1/2)   }
=\left(\frac{\sqrt{2m}}{\beta}\right)^t 2^{t/2}\Gamma(1+t/2),
\end{equation}
where we used the easy fact that $b_{0,t}=b_{0,0}t!/\beta^t=b_{0,0}\Gamma(t+1)/\beta^t$ and the duplication formula for the Gamma function. The number
$2^{t/2}\Gamma(1+t/2)$ is the $t$th moment of the \emph{Rayleigh distribution} on $[0,\infty),$ given by the distribution function $\mathbb{P}([0,x))=1-\exp\left(-x^2/2\right).$ As the Rayleigh distribution is also uniquely determined by its moments \cite{Simon98}, we have rederived the limiting distribution of the final altitude of a meander with zero drift as was previously done in \cite{BanFla02}. Notice that this was solely derived from the functional equation.

In Section \ref{sec:ProofLemmas} we shall infer from a separate treatment of the Bernoulli case that the numbers $Q_{n,t}=\frac{ b_{n,t} \beta^{n+t} }{b_{0,0}n!t!}$ uniquely determine  the joint distribution of area and final point of the Brownian meander. The $Q_{n,t}$ are precisely those defined in Theorem \ref{theo:BrownianMeander} and hence the general ``zero drift" case of Theorem \ref{theo1} can be viewed as a corollary of Theorem \ref{theo:BrownianMeander}.\qed

%

%
%

%
%
%
%
%
\section{Concentrated $Z_m$: positive drift}

Finally, we treat the case where $\Phi(z,1)$ does have a simple zero $0<z_0<\rho.$ The following recursive description of the leading singular behaviour of $F^{n,t}(1)$ at $z_0$ can easily be proven along the same lines as in the previous section. The dominant singularity in this case turns out to be a pole, more precisely 
\[
F^{n,t}(1)\sim \frac{e_{n,t}}{(1-z/z_0)^{2n+t+1}},
\]
where the numbers $e_{n,t}$ are given by the recursion
\[
e_{n,t}=t\gamma e_{n,t-1}+ne_{n-1,t+1}.
\]
The constant $\gamma$ is given as in Theorem \ref{generaltheo},
\[
\gamma=\frac{S^\prime(z_0,1)}{ S(z_0,1)+z_0\left.\frac{\partial}{\partial z} S(z,1)\right |_{z=z_0} },
\]
and the initial conditions are
\[
e_{0,0}=\left.\frac{\det {\cal N}(z,1;Y) }{Q(z,1) \det{\cal M}  }\right |_{z=z_0}
\]
 and $e_{n,t}=0$ for $n<0$ or $t<0.$ In particular a direct computation gives $e_{1,0}=e_{0,0}\gamma$ and $e_{2,0}=6e_{0,0}\gamma^2.$ An application of the Transfer theorem gives
\[
\mathbb{E}\left(Z_m \right)\sim\frac{e_{1,0}m^2}{e_{0,0}\Gamma(3)}=\frac{\gamma}{2}m^2,\;\mathbb{E}\left(Z_m^2 \right)=\frac{e_{2,0}m^4}{e_{0,0}\Gamma(5)}\sim \frac{\gamma^2}{4}m^4\sim \mathbb{E}\left(Z_m \right)^2,
\]
and hence the variance $\mathbb{V}(Z_m)$ is $o\left( \mathbb{E}\left(Z_m \right)^2 \right).$ By Chebyshev's inequality, the claimed concentration property follows. \qed
%
%
%
%
%
%
%
\section{The Bernoulli case and proof of Theorem \ref{theo:BrownianMeander}  }\label{sec:ProofLemmas}

The Bernoulli case plays a significant role in our derivations. In particular, we have weak convergence of the Bernoulli excursions to Brownian excursion, and by a theorem of Drmota \cite{Drmota04} moment convergence for any polynomially bounded functional (like area and endpoint), which we lack a proof for in more general cases. Furthermore, we want to exploit their combinatorial simplicity in the next section to prove Theorem \ref{theo:signedAreas}.  

The numbers $a_{n,t}$ and $b_{n,t}$ from Lemmas \ref{lem:singFntu1} and \ref{lem:singFnt1} are related to those from Theorem \ref{theo:BrownianMeander} as follows,
\[
C_{n,t}=\frac{ 2^{3n+t}\beta^{n+t}}{n!t!}\frac{ a_{n,t} }{ a_{0,0} },\quad Q_{n,t}=\frac{\beta^{n+t}}{n!t!}\frac{ b_{n,t} }{ b_{0,0} }.
\]
Equation \eqref{eq:recursionant} implies the recursion relation
\begin{equation}
C_{n,t}=C_{n,t-1}+(t+2)C_{n-1,t+2},\; C_{0,0}=1,\;C_{n,t}=0,\textnormal{ if }t<0 \textnormal{ or }n<0,
\end{equation}
and similarly, from equation \eqref{eq:recbnt} follows
\begin{equation}\
Q_{n,t}=
\begin{cases}
Q_{n,t-2}+(t+1)Q_{n-1,t+1},\; \textnormal{if }t\ge 1,\\
Q_{n-1,1}+ 2\cdot 8^{-n} C_{n-1,1},\;\textnormal{if }t=0,
\end{cases}
\end{equation}
with the initial values $Q_{0,0}=Q_{0,1}=1$ and $Q_{n,t}=0$ if $t<0$ or $n<0.$  The numbers $C_{n,t}$ and $Q_{n,t}$ do not depend on the parameters of the walk. We show with the help of the convergence results for the Bernoulli walks that
\begin{equation}\label{eq:cn-11anddn0}
C_{n-1,1}=8^nK_n \textnormal{ and }Q_{n,0}=Q_n,
\end{equation}
with $K_n$ and $Q_n$ as in Definitions \ref{def:BE} and \ref{def:BM}. This settles \textit{i)} of Theorem \ref{theo:BrownianMeander}, and \textit{ii)} follows from a Theorem of Petersen \cite{Petersen82}, since the sequence $M_{n,t}$ is by construction a sequence of joint moments (as a limit of such) and the projections are uniquely determined by the moment sequences $M_{n,0}$ and $M_{0,t}.$ 

\medskip
The limiting distributions of the excursion area $X_{2m}=X_{2m}^{\textnormal{Bern} } $ and meander area $Z_m=Z_m^{\textnormal{Bern}}$ for the symmetric Bernoulli step set ${\cal S}=\{-1,1\}$ with weigths $s_{-1}=s_{1}=1$ are known to equal ${\cal BEA}$ and ${\cal BMA},$ respectively \cite{Takacs91,Takacs95}. More precisely
\begin{equation}\label{limitBern}
\frac{ X_{2m} }{ (2m)^{3/2}  }\stackrel{d}{\longrightarrow}{\cal BEA} \textnormal{ and } \frac{Z_{m}}{m^{3/2}}\stackrel{d}{\longrightarrow}{\cal BMA}.
\end{equation}
Recall that there are no Bernoulli excursions with an odd number of steps. Our above derivation expresses the same limit laws in terms of numbers $a_{n,t}=a_{n,t}^{\textnormal{Bern}}$ and $b_{n,t}=b_{n,t}^{\textnormal{Bern}},$ cf. equations \eqref{eq:asExMo} and \eqref{eq:asMeMo}. Equating the respective expressions together with the above normalisation of $a_{n,t}^{\textnormal{Bern}}$ and $b_{n,t}^{\textnormal{Bern}}$ prove assertion \eqref{eq:cn-11anddn0} and hence the general case.

One detail has still to be taken care of: The step set ${\cal S}=\{-1,1\}$ is not \emph{aperiodic}, i.e. our general assumption under which we derived the singular behaviour of the $F^{n,t}(u_1)$ and the $F^{n,t}(1)$ is not fulfilled. On the one hand, aperiodicity implies the domination property of the small branches. On the other hand it implies the uniqueness of the dominant singularity $\rho.$ The domination property is trivially fulfilled, since there is only one small branch. What demands a thorough review is the occurrence of two dominant singularities.

The step polynomial is $S(u)=u^{-1}+u$ and the only small branch of the kernel equation $1-zS(u)=0$ is
\[
u_1(z)=\frac{1-\sqrt{1-2z}\sqrt{1+2z}}{2z}
\]
having two dominant singularities $z=\rho=1/2$ and $z=-1/2,$ $u_1(\pm1/2)=\pm \tau=\pm 1.$ The computations of the leading singular behaviour of $F^{n,t}(u_1)$ at $z=\rho=1/2$ are still valid, leading to a recursion \eqref{eq:recursionant} for the $a_{n,t}^{\textnormal{Bern} },$ with the parameter $\beta=\beta^{\textnormal{Bern}}=\sqrt{2}.$ Repeating the calculation at $z=-1/2$ also gives the leading singular term of the expansion about $-1/2,$ with the \emph{very same} coefficient $a_{n,t}^{\textnormal{Bern}}.$ This implies that $G^{(n)}_0(z)$ has two dominant singularities at $z=\pm1/2$ with leading terms as derived in equation \eqref{eq:singG0n}, namely
\[
G^{(n)}_0(z)\sim \frac{\epsilon \beta a_{n-1,1}^{\textnormal{Bern}}}{2a_{0,0}^{\textnormal{Bern}}(n-1)!} \frac{1}{(1\mp 2z)^{3n/2-1/2}}\quad (z\longrightarrow \pm1/2),
\]
This has to be taken into account in the process of coefficient asymptotics. The effect of the additional singularity eventually cancels out, as is seen when computing the asymptotics of the factorial moments,
\[
\begin{split}
\frac{1}{n!}\mathbb{E}\left(\left(X_{2m}\right)_n\right)&=
\frac{1}{n!}\frac{\left[z^{2m} \right]G_0^{(n)}(z) }{\left[z^{2m} \right]G^{(0)}_0(z)}\\
&\sim \frac{-a_{n-1,1}^{\textnormal{Bern}}}{2a_{0,0}^{\textnormal{Bern}}(n-1)! }  \frac{\Gamma(-1/2)}{\Gamma(3n/2-1/2)}\frac{\left( 2^{2m}+(-2)^{2m} \right) (2m)^{3n/2-3/2} }{\left( 2^{2m}+(-2)^{2m} \right)  (2m)^{-3/2} }.   
\end{split}
\]
Since factorial and ordinary moments asymptotically coincide, equating the so obtained expression for $\lim_{m\rightarrow \infty}\mathbb{E}\left(\left((2m)^{-3/2}X_{2m}\right)^n\right)$ with the one from Definition \ref{def:BE} leads to
\[
\frac{-C_{n-1,1}}{2^{3n-1}(\beta^{\textnormal{Bern}})^n }=\frac{-a_{n-1,1}^{\textnormal{Bern}}}{2a_{0,0}^{\textnormal{Bern}}(n-1)! } =\frac{K_n2^{-n/2}}{K_0}=-2\frac{K_n}{(\beta^{\textnormal{Bern}})^n },
\]
which settles the first part of assertion \eqref{eq:cn-11anddn0}.

\medskip
We proceed with the Bernoulli meander. The computation of the leading singular term of $F^{n,t}(1)$ at $z=1/2$ leads to a recursion \eqref{eq:recbnt}, with $\beta=\beta^{\textnormal{Bern}}=\sqrt{2}.$ $F^{n,t}(1)$ is of course also singular at $z=-1/2,$ but of neglegible order. One can show by induction that $F^{n,t}(1)=O(F^{n-1,1}(u_1))$ as $z\longrightarrow   -1/2.$ We hence have by \eqref{eq:asMeMo},  \eqref{limitBern} and Definition \ref{def:BM}
\begin{equation}
\lim_{m\rightarrow\infty}\frac{1}{n!}\mathbb{E}\left(\left( \frac{Z_m}{m^{3/2}} \right)^n\right) = \frac{b_{n,0}^{\textnormal{Bern}}\Gamma(1/2)}{n!b_{0,0}^{\textnormal{Bern}}\Gamma(3/2n+1/2)}=\frac{Q_n\Gamma(1/2) 2^{-n/2} }{Q_0\Gamma(3n/2+1/2)},
\end{equation}
which finally yields 
\[
\frac{Q_{n,0}}{(\beta^\textnormal{Bern})^n}=\frac{b^\textnormal{Bern}_{n,0}}{n!b^\textnormal{Bern}_{0,0}}=\frac{Q_n}{(\beta^\textnormal{Bern})^n}.
\]
This settles the second part of assertion \eqref{eq:cn-11anddn0} and hence the proof of Theorem \ref{theo:BrownianMeander}.
\medskip

\noindent
\textbf{Remark.} The case of a more general \emph{periodic} step set can be treated similarly. The branches $u_1$ and $v_1$ coalesce in the points $(z,u)=\left(\rho e^{-2 \pi i  jc/p}, \tau e^{2 \pi i j/p}\right),$ $j=0,\ldots,p-1.$ Moreover, one has to argue more carefully for the separation property.

\section{Example: Column convex polygons}\label{columnconvex}

For many polygon models on the square lattice the area distribution in a uniform fixed-perimeter ensemble has been studied and limit laws have been derived with the help $q$-algebraic functional equations satisfied by the area and perimeter generating function \cite{Richard02,Richard09II,SchRicTha10}. This approach did not work out for column convex polygons due to lack a of ``symmetry in horizontal and vertical edges". Recall that a polygon is called column convex if every intersection with a vertical line is convex. We will prove that the limiting area distribution is ${\cal BEA}.$ 

Denote by $F(z,q,u)$ the generating function of column convex polygons, in which $z$ marks the half-perimeter, $q$ the area and $u$ the height of the rightmost column. These polygons can be built up by adding one column at a time, leading to a functional equation \eqref{eq:generalfeq} with \cite{Bousquet96} 
\[
S(z,u)=\frac{u^2(1-z)^2}{(1-u  )^2(1-zu)^2},\quad r_0(z,u)=\frac{zu^2(2z-zu-1)}{(1-u )^2(1-zu)},\quad r_1(z,u)=\frac{zu}{1-u}
\]
and $W(z,u)=\frac{z^2u}{1-zu}.$ Furthermore, $G_0(z,q)=F(z,q,1)$ and $G_1(z,q)=\left.\frac{\partial}{\partial u}F(z,q,u)\right |_{u=1}.$ Denominators can be cleared with $Q(z,u)=(1-u  )^2(1-zu)^2$ and we see that $\Phi(z,u)=(1-zS(z,u))Q(z,u),$ $t_i=r_i\cdot Q,$ and $W\cdot Q$ are polynomials. $\Phi$ is a polynomial degree 4 in $u$ and since $\Phi(0,u)=(1-u)^2,$ there are two branches with Puiseux expansions $1\pm \sqrt{z}+\ldots$ about $0.$ The ``$+$" branch plays the role of $u_1(z)$ and the other that of $u_2(z).$ The point $(\rho,\tau)$ is found by computing the resultant of $\Phi(z,u)$ and $\Phi^\prime(z,u).$ Critical values for $z$ turn out to be $0,$ $\pm1,$ $3+2\sqrt{2}$ and $\rho=3-2\sqrt{2}.$ The value $\tau=1+\sqrt{2}$ is the unique double zero of $\Phi(\rho,u),$ and $S^{\prime\prime}(\rho,\tau)>0.$ Within $0<|z|\leq\rho$ none of $u_1,$ and $u_2$ solves $Q(z,u)=0,$ since $\Phi(z,1)=0$ only for $z=0,1$ and $\Phi(z,1/z)=0$ only for $z=1.$ 
Both $u_1$ and $u_2$ can be continued analytically along $(0,\rho)$ with $u_1(z)>1$ and $u_2(z)<1.$ The latter is true close to $z=0$ and remains true throughout $(0,\rho]$ for continuity reasons, since the equation $\Phi(z,1)=0$ has only the solutions $z=0,1.$ It follows that $u_1$ has a square root singularity at $\rho$ while $u_2$ is analytic throughout $|z|\leq \rho.$ Finally we look at $\det {\cal M}=t_0(z,u_1)t_1(z,u_2)-t_0(z,u_2)t_1(z,u_1).$  We have 
\[
\det {\cal M}=z^2(u_1-u_2)u_1u_2(1-zu_1)(1-zu_2)\varphi(u_1,u_2)
\]
and it remains to argue why $\varphi(u_1,u_2)=(z(u_1+u_2)+z(z-2)u_1u_2-2z+1)$ remains nonzero in $0<|z|\leq\rho.$ Assume that $\varphi(u_1,u_2)=0$ for some $z_0$ and solve this equation for $u_1.$ Substitute the so-obtained expression for $u_1$ in terms of $z_0, u_2$ into the equation $\Phi(z_0,u_1(z_0))=0,$ which yields a rational expression in $z_0$ and $u_2$ whose numerator $\Psi(z_0,u_2)$ is a polynomial in $z_0,u_2$ of degree 4 in $u_2$ with integer coefficients. So $(z_0,u_2(z_0))$ is a common solution of both $\Phi(z,u)=0$ and $\Psi(z,u)=0.$  A possible $z_0$ is hence to be sought among the zeroes of the resultant of $\Phi$ and $\Psi,$ which are $z=0,$ $z_0=0.167\ldots<\rho$ and some other values of $z$ with $|z|>\rho.$ The common solution of $\Phi(z_0,u)=0$ and $\Psi(z_0,u)=0$ turns out to be $2.116\ldots >1$ and hence cannot coincide with $u_2(z).$ So $\det {\cal M}\neq 0$ in the domain in question and the assumptions of Section \ref{sec:analyticassumptions} are satisfied. 

\section{Unconstrained walks: joint distribution of final altitude and signed areas}
In this section we apply the result on the joint distribution of meander area and final altitude to get this joint distribution in a number of other cases. We have to restrict ourselves to the symmetric Bernoulli case as we want to use a decomposition which is restricted to that case, namely cutting the path whenever it \emph{crosses} the $x$-axis, i.e.  changes the sign. This yields a decomposition into positive and negative excursions which fails for general step sets. However, the rescaled Bernoulli walk converges to Brownian motion and the results carry over to that continuous limit making the considerations worthwhile. We consider the joint distribution of the final altitude and absolute area, area of the positive part and area of the positive and negative parts. The computations are pretty repetitive in the three cases, the most annoying of which is certainly the joint distribution of final altitude and the two signed areas. We demonstrate this latter one in more detail and only indicate the remaining. In this way rederive a number of formulae for the various Brownian areas as collected in the survey \cite{Janson07}. The joint moments with final altitude turn out to be simple generalisations of these.

A bridge is a walk with steps from ${\cal S}=\{-1,1\}$ that starts and ends at $0.$ Denote by $B(z,q_-,q_+)$ its generating function, where $z$ marks the number of steps, $q_+$ the positive and $q_-$ the negative area. A bridge has either an even or an odd number of sign changes. In the even case it can be decomposed into a non-empty initial excursion and a sequence of pairs consisting of a non-empty positive and a non-empty negative excursion. In terms of generating functions this reads after some simplifications
\begin{equation}
B(z,q_+,q_-)=:B(q_+,q_-)=\frac{G_0(z,q_+) G_0(z,q_-)}{G_0(z,q_+)+G_0(z,q_-)-G_0(z,q_+)G_0(z,q_-)},
\end{equation}
in particular $B(1,1)=G_0^{(0)}(z)\left(2-G_0^{(0)}(z)\right)^{-1}.$ Similarly, we decompose unconstrained walks. A walk is either empty, or it can be decomposed in a (possibly empty) initial excursion, followed by a (possibly empty) sequence of pairs of non-empty excursions as above, followed by a \emph{non-empty} meander. Some simplifications finally yield for the generating function $W(q_+,q_-,u):=W(z,q_+,q_-,u)$ the representation
\begin{equation}
W(q_+,q_-,u)=\frac{G_0\left(z,q_-\right )F\left(z,q_+,u\right ) + G_0\left(z,q_+\right )F\left(z,q_-,u^{-1}\right )    }{G_0(z,q_+)+G_0(z,q_-)-G_0(z,q_+)G_0(z,q_-)}-B(z,q_+,q_-).
\end{equation}
For convenience we recall the singular behaviour of $G^{(n)}_0(z)$ and $F^{n,t}(1)$ at $z=1/2$ in the Bernoulli case, namely
\begin{equation}\label{G0Gn}
G^{(0)}_0(z)\sim 2-2\sqrt{2}\sqrt{1-2z},\quad G^{(n)}_0(z)\sim \frac{2^{-n/2} \cdot 4\sqrt{2}K_n n! }{  (1-2z)^{3n/2-1/2}  } \quad (z\longrightarrow 1/2).
\end{equation}
and
\begin{equation}\label{FntBern}
F^{n,t}(1)\sim \frac{ n!t! 2^{(1-n-t)/2} Q_{n,t}  }{ (1-2z)^{3n/2+t/2+1/2} }.
\end{equation}
\subsection{Signed areas and final altitude}


As in the previous sections, we introduce shorthand notation for the mixed derivatives evaluated at $q_+=q_-=u=1,$ namely
\begin{equation}
W^{k,l,t}=W^{k,l,t}(z)=\left(\frac{\partial}{\partial q_+}\right)^k\left(\frac{\partial}{\partial q_-}\right)^l\left(\frac{\partial}{\partial u}\right)^tW(1,1,1),
\end{equation}
analogously $B^{k,l}$ and $R^{k,l}$ for $R(q_+,q_-)=\left(G_0(z,q_+)+G_0(z,q_-)-G_0(z,q_+)G_0(z,q_-)\right)^{-1}$ is defined. 

\noindent
A twofold application of (a univariate version of) Fa\`a di Bruno's formula \cite{ConSav96}  yields 
\begin{equation}\label{faadibruno4}
\begin{split}
\left(\frac{\partial}{\partial q_+}\right)^k\left(\frac{\partial}{\partial q_-}\right)^lR(q_+,q_-)=\sum_{(\pi,\eta)}\frac{ (|\pi|+|\eta|)! \left( G_0(z,q_-)-1\right)^{|\pi|+|\eta|}    }{ R(q_+,q_-)^{|\pi|+|\eta|+1}   }\\
\times\prod_{  \begin{smallmatrix}  b\in \pi \\ c \in \eta  \end{smallmatrix} } \left(\frac{\partial}{\partial q_+}\right)^{|b|}G_0(z,q_+)\left(\frac{\partial}{\partial q_-}\right)^{|c|}G_0(z,q_-)+\ldots , 
\end{split}
\end{equation}
where the sum runs over pairs of partitions $(\pi,\eta)$ of a $k$- resp. $l$-set and the product over the blocks of the respective partitions. By $| \cdot |$ we denote the cardinality of a set. The omitted trailing terms are of smaller asymptotic order as $q_+=q_-=1$ and $z\longrightarrow 1/2.$ Observing that $R(1,1)=\left(G_0^{(0)}(z)\left(2-G_0^{(0)}(z)\right)\right)^{-1}$ and $G_0^{(0)}(1/2)=2$ this yields together with the asymptotic representations \eqref{G0Gn}
\begin{equation}\label{Rkl}
R^{k,l}\sim\frac{2^{-(k+l)/2}   } {4\sqrt{2}(1-2z)^{3k/2+3l/2+1/2}  } D_{k,l}^\pm
\end{equation}
where the numbers $D_{k,l}^\pm$ can by \eqref{faadibruno4} be expressed as
\begin{equation}\label{Dklpm}
 D_{k,l}^\pm=  \sum_{\pi,\eta}  (|\pi|+|\eta|)! \prod_{  \begin{smallmatrix}  b\in \pi \\ c \in \eta  \end{smallmatrix} } K_{|b|}|b|!K_{|c|}|c|!=\left[x^ky^l \right] \frac{1}{1-\sum_{n\ge1}K_n(x^n+y^n)  }.
\end{equation}
The numbers $D^\pm_{k,l}$ are related to the joint moments of the signed areas of the Brownian Bridge, see \cite[eqs. (251) and (252)]{Janson07} and \cite{PerWel96}. Indeed, we further find that $B^{k,l}(z)\sim 4R^{k,l}(z)$ as $z\longrightarrow 1/2.$ 
Applying the product rule and regarding only the dominant terms we obtain
\begin{equation}\label{Wklt} 
\begin{split}
W^{k,l,t}&\sim\sum_{i=0}^k{k \choose i}F^{k-i,t}(1)G^{(0)}_0R_{i,l} +(-1)^t\sum_{j=0}^l{l \choose j}F^{l-j,t}(1)G^{(0)}_0R_{k,j}\\
&\sim \frac{2^{-(k+l+t)/2}k!l!t!  } {(1-2z)^{3k/2+3l/2+1}  } L^{\pm}_{k,l,t} ,
\end{split}
\end{equation}
where by equations \eqref{Rkl}, \eqref{G0Gn} and \eqref{FntBern} we have the following representation of the $L^{\pm}_{k,l,t} $ 
\begin{equation}
L^{\pm}_{k,l,t}=\frac{1}{2}\sum_{i=0}^kQ_{k-i,t}D^{\pm}_{i,l}  + \frac{(-1)^t}{2}  \sum_{j=0}^lQ_{l-j,t}D^{\pm}_{k,j},\\
\end{equation}
the factor $(-1)^t$ resulting from
\begin{equation}\label{Fntuminus1}
\left .  \left(\frac{\partial}{\partial q}\right)^n\left(\frac{\partial}{\partial u}\right)^t F(z,q,u^{-1})\right |_{u=q=1}=(-1)^t\left(F^{n,t}(1)+t(t-1)F^{n,t-1}(1)+\ldots \right).
\end{equation}
We hence have moment convergence for the rescaled random variables,
\[
\mathbb{E}\left( \left(\frac{Z^+_m}{m^{3/2}} \right)^k \left(\frac{Z^-_m}{m^{3/2}} \right)^l  \left( \frac{H^m}{m^{1/2}}\right)^t \right)\longrightarrow \frac{k!l!t!2^{-(k+l+t)/2}}{\Gamma(3k/2+3l/2+t/2+1) } L^{\pm}_{k,l,t}.
\] 
In terms of generating functions we have the relation
\begin{equation}
\begin{split}
\sum_{k,l,t}L^{\pm}_{k,l,t}x^ky^lw^t=\frac{\sum_{n,t\ge0}Q_{n,t}\left(x^nw^t+y^n(-w)^t   \right)     }{-2\sum_{n\ge0}K_n(x^n+y^n)  },
\end{split}
\end{equation}
which for $w=0$ specialises to \cite[eq. (275)]{Janson07}. Furthermore we obtain the joint distribution of the area of the positive part and terminal altitude by simply specialising to $l=0,$ which, for $t=0,$ gives \cite[eq. (298)]{Janson07}. The numbers $D^{\pm}_{0,k}=D^{\pm}_{k,0}$ are referred to as $D_k^{+}$ in \cite{Janson07}.

\subsection{Absolute area and final altitude}

This is the case $q_+=q_-=q,$ in which the above formulae for the bridge and walk generating functions simplify to
\[
B(z,q)=\frac{G_0(z,q)}{2-G_0(z,q)},\quad W(q,u)=\frac{F(z,q,u)+F(z,q,u^{-1})}{2-G_0(z,q)}-B(z,q).
\]
By an application of Fa\`a di Bruno's formula we find similar to the above derivation that 
\[
\frac{1}{2} \left(\frac{\partial}{\partial q}\right)^k B(z,1)\sim \left.\left(\frac{\partial}{\partial q}\right)^k \frac{1}{2-G_0(z,q)}\right|_{q=1}  \sim \frac{2^{-k/2} D_k }{2\sqrt{2}(1-2z)^{3k/2+1/2}}.
\]
The numbers $D_k$ are given by
\[
D_k:=\frac{1}{k!}\sum_{\pi}2^{ |\pi|} |\pi|!   \prod_{b\in \pi} K_{|b|}|b|! =\left[x^k \right] \frac{1}{1-2\sum_{n\ge1}K_nx^n}=\left[x^k \right] \frac{K_0}{\sum_{n\ge0}K_nx^n}.
\]
where $\pi$ runs through the partitions of a $k$-set and $b$ through the blocks of $\pi.$ The numbers $D_k$ are related to the moments of the absolute area of the Brownian bridge, cf. \cite[eqs. (135) and (287)]{Janson07} and \cite{Kac46,PerWel96,Takacs93}. Comparing the generating functions of the numbers $D_k$ and $D^{\pm}_{i,j}$ from \eqref{Dklpm}  we see $D_k=\sum_{i=0}^kD^\pm_{k-i,i},$ cf. \cite[eq. (254)]{Janson07}.
The product rule then gives for even orders $2t$ 
\[
\begin{split}
\left(\frac{\partial}{\partial q}\right)^n\left(\frac{\partial}{\partial u}\right)^{2t} W(1,1)&\sim \sum_{k=0}^n {n \choose k}2F^{n-k,2t}(1)\frac{k!D_k2^{-k/2}}{2\sqrt{2}(1-2z)^{3k/2+1/2}}\\
&=\frac{n!(2t)!2^{-n/2-t}}{(1-2z)^{3n/2+t+1} } L_{n,2t}
\end{split}
\]
with $L_{n,2t}=\sum_{k=0}^n Q_{n-k,2t} D_k.$ We have the generating function identity
\[
\sum_{n,t\ge 0}L_{n,2t} x^ny^{2t}  =\frac{K_0 \sum_{n,t\ge 0}Q_{n,2t} x^ny^{2t}  }{\sum_{n\ge0}K_nx^n}
\]
By singularity analysis we see that the joint moments of the rescaled random variables $X_m$ of absolute area and $H_m$ of final altitude converge to a finite limit, namely
\[
\mathbb{E}\left( \left(\frac{X_m}{m^{3/2}} \right)^n  \left( \frac{H^m}{m^{1/2}}\right)^{2t} \right)\longrightarrow \frac{n!t!2^{-n/2-t}}{\Gamma(3n/2+t+1) } L_{n,2t}.
\] 
For odd $t,$ the term $F^{n,t}(1)$ cancels out and the joint moments are $O(m^{-1/2})$ and hence the limiting moments are $0,$ in accordance with the obvious invariance of the distribution of the both variables under reflection of the Bernoulli walk. Hence we set $L_{n,2k+1}=0$ for $k\ge 0.$ For $t=0$ we retrieve \cite[formula (290)]{Janson07}.

\section{Conclusion}

We have proven a limit theorem for random variables defined in terms of coefficients of power series which in turn are determined by a ``$q$-shift" of a linear equation with one catalytic variable. The main application thereof are the limit distributions for the area under discrete excursions and meanders with an (almost) arbitrary finite set of steps, providing some new results and re-deriving some classical ones on walks on the integers. As another combinatorial application of this result we prove an area limit law for column convex polygons, which has so far only been done numerically.

As a bonus we computed the joint distribution of several Brownian areas and the height of the endpoint in terms of their moments. In particular Theorem \ref{theo:BrownianMeander} yields new recursion relations for the moments of ${\cal BEA}$ and ${\cal BMA}.$ 

For meanders with a positive drift we know from the theory of random walks that the limit law for the centered and normalised area random variable is Gaussian, and we believe that this also holds in the more general setting. Our method, however, would demand a considerable modification to prove this, since knowledge of only the dominant term of the singular expansions is no longer sufficient.

According to our assumptions \ref{sec:analyticassumptions} we need a good knowledge of the solutions of the kernel equation in order to make Theorem \ref{generaltheo} effective, which is similar for \cite[Theorem 1.5]{Richard09}.  Our assumption 5 on the determinant in Section \ref{sec:analyticassumptions} may be difficult to check, as the computations for column convex polygons may suggest.  

\section*{Acknowledgements}
This research was conducted during a stay of the author at RMIT University and the University of Melbourne. He would like to thank both institutions for their hospitality.
He gratefully acknowledges joint financial support by RMIT University, The ARC Centre of Excellence for Mathematics and Statistics of Complex Systems and the Alexander von Humboldt Foundation within the Feodor Lynen programme. In particular he wants to thank Robin Hill and Tony Guttmann for pulling the attached strings.


\begin{thebibliography}{99}
\bibitem{BanFla02}
C. Banderier, P. Flajolet, Basic analytic combinatorics of directed lattice paths, \emph{Theor. Comput. Sci.} \textbf{281} (2002), 37-80.

\bibitem{BanGit06}
C. Banderier and B. Gittenberger, Analytic combinatorics of lattice paths: Enumeration and asymptotics for the average area, \emph{Discrete Math. Theor. Comput. Sci.} \textbf{AG} (2006), pp. 345-355.

\bibitem{Bousquet96}
M. Bousquet-M{\'e}lou, A method for the enumeration of various classes of column-convex polygons , \emph{Discrete Math.} \textbf{154} (1996), 1-25.


\bibitem{BouPet00}M. Bousquet-M{\'e}lou, M. Petkov$\check{\text{s}}$ek, Linear recurrences with constant coefficients: the multivariate case, \emph{Discrete Mathematics} \textbf{225} (2000), 51-75.

\bibitem{BouJeh06}
M. Bousquet-M{\'e}lou, M. Jehanne, Polynomial equations with one catalytic variable, algebraic series and map enumeration, \emph{J. Comb. Theory, Series B}
\textbf{96}, 5 (2006), 623-672.

\bibitem{Bressoud99}
D.M. Bressoud, \emph{Proofs and Confirmations}, Cambridge: Cambridge University Press (1999).

\bibitem{Chung74}
K. L. Chung, \emph{A course in probability theory} New York - London: Academic Press (1974).

\bibitem{ConSav96}
G. M. Constantine, T. H. Savits, A Multivariate Faa di Bruno Formula with Applications, \emph{Trans. Amer. Math. Soc.} \textbf{348} (1996), 503- 520

\bibitem{Drmota04}
M. Drmota, Stochastic analysis of tree-like data structures, \emph{Proc. R. Soc. Lond. A.} \textbf{460} (2004), 271-307.

\bibitem{Duchon99}
P. Duchon, $q$-grammars and wall polyominoes, \emph{Annals of Combinatorics} \textbf{3} (1999),311-321.


\bibitem{DurIglMil77}
R.T. Durrett, D.L. Iglehart, D.R. Miller, Weak convergence to Brownian meander and Brownian excursion, 
\emph{Ann. Probability} \textbf{5}, 1 (1977), 117-129. 

\bibitem{Durrett80}
R. Durrett, Conditioned limit theorems for random walks with negative drift, \emph{Z. Wahrscheinlichkeitstheorie verw. Geb.}, \textbf{52} (1980), 277-287. 

\bibitem{FlaLou01}
P. Flajolet, G. Louchard, Analytic variations on the Airy distribution, \emph{Algorithmica} \textbf{31} (2001), 361-377.

\bibitem{FlaOdl90}
P. Flajolet, A. M. Odlyzko, Singularity analysis of generating functions, \emph{SIAM J. Discrete Math.} \textbf{3} (1990),216-240.

\bibitem{FlaSed09}
P. Flajolet, R. Sedgewick, \emph{Analytic Combinatorics}, Cambridge: Cambridge University Press (2009).

\bibitem{Hille62}
E. Hille, \emph{Analytic Function Theory}, two volumes, Blaisdell Publishing Company, Waltham, (1962).

\bibitem{Iglehart74}
D.L. Iglehart, Functional central limit theorems for random walks conditioned to stay positive, \emph{Ann. Probability} \textbf{2}, 4 (1974), 608-619. 

\bibitem{Janson07}
S. Janson, Brownian excursion area, WrightÕs constants in graph enumeration, and other Brownian areas, \emph{Probability Surveys}
\textbf{4} (2007), 80-145,
DOI: 10.1214/07-PS104

\bibitem{Kac46}
M. Kac, On the average of a certain Wiener functional and a related limit
theorem in calculus of probability, \emph{Trans. Amer. Math. Soc.} \textbf{59} (1946), 401-414.

\bibitem{Kaigh76}
W. D. Kaigh, An invariance principle for random walk conditioned by a late return to zero, \emph{Ann. Probability} \textbf{4},1 (1976), 115-121.

\bibitem{Kao78}
P. Kao, Limiting Diffusion for Random Walks with Drift Conditioned to Stay Positive, \emph{J. Appl. Prob.} \textbf{15} (1978), 280-291.

\bibitem{Murota00}
K. Murota, \emph{Matrices and Matroids in Systems Analysis},  Springer Verlag, Berlin/Heidelberg/New York (2000).

\bibitem{Nguyen03}
M. Nguy$\hat{\textnormal{e}}$n Th$\hat{\textnormal{e}}$, Area of Brownian Motion with Generatingfunctionology, \emph{Discrete Mathematics and Theoretical Computer Science} \textbf{AC} (2003), 229-242

\bibitem{PerWel96}
M. Perman, J.A. Wellner, On the distribution of Brownian areas, \emph{Ann. Appl. Probab.} \textbf{6}, 4 (1996), 1091-1111.

\bibitem{Petersen82}
L.C. Petersen, On the relation between the multidimensional moment problem and the one-dimensional moment problem, \emph{Math. Scand.} \textbf{51} (1982), 361-366.

\bibitem{Richard02}
C. Richard, Scaling behaviour of two-dimensional polygon models, \emph{J. Stat. Phys} \textbf{108} (2002), 459-493



\bibitem{Richard09} 
C. Richard, On $q$-functional equations and excursion moments, \emph{Discrete Math.} \textbf{309} (2009), 207-230.

\bibitem{Richard09II}
C. Richard, Limit distributions and scaling functions, In \emph{Polygons, Polyominoes and Polycubes}, Springer, Berlin/Heidelberg (2009), 247-299.

\bibitem{Scheidemann05}
V. Scheidemann, \emph{Introduction to Complex Analysis in Several Variables}, Birkh{\"a}user Verlag, Basel/Boston/Berlin (2005).

\bibitem{SchRicTha10}
U. Schwerdtfeger, C.Richard, B. Thatte, Area limit laws for symmetry classes of staircase polygons, \emph{Combin. Prob. Comput.} \textbf{19} (2010), 441-461.

\bibitem{Simon98}
B. Simon, The Classical Moment Problem as a Self-Adjoint
Finite Difference Operator, \emph{Advances in Mathematics} \textbf{137} (1998), 82-203.

\bibitem{Takacs91}
L. Tak\'acs, A Bernoulli excursion and its various applications, \emph{Adv. Appl. Probab.} \textbf{23} (1991), No. 3, 557-585.

\bibitem{Takacs93}
L. Tak\'acs, On the distribution of the integral of the absolute value of the
Brownian motion, \emph{Ann. Appl. Probab.} \textbf{3}, 1 (1993), 186-197.


\bibitem{Takacs95}
L. Tak\'acs, Limit Distributions for the Bernoulli meander, \emph{J. Appl. Prob.} \textbf{32} (1995), 375-395.


\end{thebibliography}

\end{document}